\documentclass[10pt]{amsart}
\usepackage{amsmath,mathrsfs,amsthm}
\usepackage{txfonts,rotating,tikz}
\usepackage{esvect}

\usepackage[all]{xy}
\usepackage{nicefrac}

\usepackage[all]{xy}
\usepackage{color} 
\usepackage{hyperref}

\newcommand{\doi}[1]{\href{https://doi.org/#1}{\texttt{doi:#1}}}
\newcommand{\arxiv}[1]{\href{https://arxiv.org/abs/#1}{\texttt{arXiv:#1}}}

\newtheorem{theorem}{Theorem}[section]

\newtheorem{lemma}[theorem]{Lemma}
\newtheorem{proposition}[theorem]{Proposition}
\newtheorem{corollary}[theorem]{Corollary}

\newtheorem{remark}[theorem]{Remark}

\newtheorem{notation}[theorem]{Notation}
\newtheorem{definition}[subsection]{Definition}

\DeclareMathOperator{\Hom}{\mathrm{Hom}}

\DeclareMathOperator{\ad}{\mathrm{ad}}
\DeclareMathOperator{\Res}{\mathrm{Res}}

\DeclareMathOperator{\Id}{\mathrm{Id}}

\DeclareMathOperator{\low}{\mathrm{low}}

\newcommand{\fh}{\mathfrak{h}}

\newcommand{\fb}{\mathfrak{b}}
\newcommand{\fg}{\mathfrak{g}}

\newcommand{\CB}{{\mathcal {B}}}

\newcommand{\hB}{{\hat{\mathcal {B}}}}
\newcommand{\Bo}{{\mathcal{B}}_{\underline{0}} }

\newcommand{\bbP}{{\mathbb {P}}}

\newcommand{\bbR}{{\mathbb {R}}}

\newcommand{\beq}{\begin{equation}}
\newcommand{\eeq}{\end{equation}}
\newcommand{\bt}{\begin{theorem}}
\newcommand{\et}{\end{theorem}}
\newcommand{\bde}{\begin{definition}}
\newcommand{\ede}{\end{definition}}
\newcommand{\bpr}{\begin{proposition}}
\newcommand{\epr}{\end{proposition}}
\newcommand{\ble}{\begin{lemma}}
\newcommand{\ele}{\end{lemma}}
\newcommand{\bco}{\begin{corollary}}
\newcommand{\eco}{\end{corollary}}
\newcommand{\bre}{\begin{remark}}
\newcommand{\ere}{\end{remark}}
\newcommand{\bpf}{\begin{proof}}
\newcommand{\epf}{\end{proof}}

\def\Square{\Box}

\newcommand{\bc}{\mathbb{C}}

\newcommand{\bz}{\mathbb{Z}}

\begin{document}

\title{ Lie Algebra Cohomology of the positive part of Twisted Affine Lie Algebras}

\author{Jiuzu Hong}
\address{Department of Mathematics, University of North Carolina at Chapel Hill, Chapel Hill, NC 27599-3250, U.S.A.}
\email{jiuzu@email.unc.edu}
\author{Shrawan Kumar}
\address{ 
Department of Mathematics, University of North Carolina at Chapel Hill, Chapel Hill, NC 27599-3250, U.S.A.}
\email{shrawan@email.unc.edu}

\maketitle
\begin{abstract}
The explicit Verlinde  formula for the dimension of conformal blocks, attached to a marked projective curve $\Sigma$, a simple Lie algebra $\fg$ over $\mathbb{C}$  and integrable highest weight modules of a fixed central charge 
of the corresponding affine Lie algebra $\hat{L}(\fg)$ attached to the marked points, requires (among several other important ingredients) a Lie algebra cohomology vanishing result due to C. Teleman for the positive part $\hat{L}^+(\fg)$
with coefficients in the tensor product of an integrable highest weight module with copies of finite dimensional evaluation modules. The aim of this paper is to extend this result of Teleman to a twisted setting where  $\fg$ is endowed with a special automorphism $\sigma$ and the curve $\Sigma$ is endowed with the action of  $\sigma$. In this general setting, the affine Lie algebra gets replaced by twisted affine Lie algebras. The crucial ingredient (as in Teleman) is to prove a certain Nakano Identity. 

\end{abstract}

\section{Introduction} We take the base field in this paper to be the field of complex numbers $\mathbb{C}$. 
The dimension of the space of conformal blocks can be computed by the celebrated Verlinde formula  (cf. \cite[Theorem 4.2.19]{Kbook2}).
 This formula was originally conjectured by Verlinde in Conformal Field Theory. It was mathematically derived by combining the efforts of several mathematicians (cf. \cite[4.C]{Kbook2}). Some of the main ingredients going into its proof consisted of reducing the problem from a genus $g$ curve $\Sigma$ with any number of marked points to a genus $0$ curve (i.e. $\mathbb{P}^1$) with three marked points (by using the Factorization Theorem and the existence of a flat projective connection on the bundle of conformal blocks over the moduli space of curves  \cite{TUY}). The problem for three marked points on   $\mathbb{P}^1$ is encoded in  so called the {\it Fusion Algebra}. Faltings obtained a certain result on the classification of characters of the fusion algebra for the classical simple Lie algebras $\fg$  as well as $\fg$ of type $G_2$. This allows one to derive  the  explicit dimension formula for these Lie algebras. Then, Teleman proved the cohomology vanishing result mentioned in the Abstract (see below for a precise statement) for any simple $\fg$, thereby  one obtains  the  explicit dimension formula for the space of conformal blocks for any simple $\fg$.

Let $\Gamma$ be a finite group acting on a simple Lie algebra $\mathfrak{g}$   and a projective algebraic curve $\Sigma$. One may define the space of twisted conformal blocks in this equivariant setting.  A generalization of the work of Tsuchiya-Ueno-Yamada \cite{TUY} in this setting was developed  in \cite{Da,HK1}. Subsequently, in \cite{HK2} the authors proved that the calculation of the dimension of twisted conformal blocks can be reduced to the situation that $\Gamma$ acts on $\mathfrak{g}$ by diagram automorphisms. Moreover, an explicit  formula for the dimension of the twisted conformal blocks (akin to the classical Verlinde dimension formula)
can be derived assuming a twisted analogue of Teleman's vanishing theorem on the  Lie algebra (co)homology
(cf. \cite[Theorem 6.9]{HK2}).  It should be mentioned that, using the machinery of modular tensor categories, under some restrictions, the dimension formula for the  twisted conformal blocks was obtained in \cite{DM} earlier.  However, our approach to the same problem is quite different. In fact, combining the works of \cite{HK1,HK2,DM}, one can derive the dimension formula for the twisted conformal blocks for a general $\Gamma$-curve $\Sigma$ and any $\Gamma$ action on $\mathfrak{g}$.  

Let $\Gamma=\langle \sigma \rangle$ be the cyclic group acting on $\mathbb{P}^1$ with ramified points only at $\{0,\infty\}$, and $\sigma$  acting on $\mathfrak{g}$ as a special automorphism (cf. Definition \ref{sect_special_aut}).  Let $\hat{L}(\mathfrak{g},\sigma)$ be the associated twisted affine Lie algebra, the central extension of the twisted loop algebra $\left(\mathfrak{g}\otimes \mathbb{C}((t))\right)^\sigma$ and let  $L^+(\mathfrak{g},\sigma): = (t\mathfrak{g}[t] )^\sigma$ be its Lie subalgebra (called the positive part of $\hat{L}(\mathfrak{g},\sigma)$).
  Let $\mathscr{H}_c(\lambda)$ be the integrable highest weight module of $\hat{L}(\mathfrak{g},\sigma)$ of level $c$ and highest weight $\lambda$. We choose $s+1$ points $\vec{z}$ in $\mathbb{C}\subset \mathbb{P}^1$ such that their $\Gamma$-orbits are distinct, where $\vec{z}=(z_0=0, z_1,\dots, z_s)$. At each point $z_i$, we attach the finite dimensional irreducible representation $V(\mu_i)$ of the fixed point Lie subalgebra $\mathfrak{g}^{\Gamma_i}$ with highest weight $\mu_i$,  
   where $\Gamma_i$ is the stabilizer group of $\Gamma$ at $z_i$.  Thus, with our choice, $V(\mu_0)$ is an irreducible representation of 
   $\fg^\sigma$, whereas $V(\mu_i)$ (for $1\leq i\leq s$) is an irreducible representation of $\fg$. 
   The choice of $z_i$  gives rise to an evaluation representation $V(\mu_i)_{z_i}$ of $L^+(\fg,\sigma)$ at $z_i$.  Thus, $L^+(\mathfrak{g},\sigma)$ naturally acts on the tensor product
$$ \hat{\mathscr{H}}_c(\lambda)\otimes V(\vec{\mu})_{\vec{z}}:= \hat{\mathscr{H}}_c(\lambda)\otimes V(\mu_0)_{z_0}\otimes \cdots  \otimes V(\mu_s)_{z_s} , $$ 
where $\hat{\mathscr{H}}_c(\lambda)$ is the   direct {\it product} of the energy eigenspaces of $\mathscr{H}_c(\lambda)$.

Let $H^q(L^+(\mathfrak{g},\sigma),   \hat{\mathscr{H}}_c(\lambda)\otimes V(\vec{\mu}) )$ be the $q$-th Lie algebra cohomology of $L^+(\mathfrak{g},\sigma)$ with coefficients in $\mathscr{H}_c(\lambda)\otimes V(\vec{\mu})_{\vec{z}}$.  Our main result of this paper is Theorem \ref{thm3.15}, which asserts the following. 
\vskip1ex

\noindent
{\bf Theorem.} {\it Let $\sigma$ be a special automorphism of $\fg$ of order $m$ and let $(\fg, m)$ be different from $(D_4, 3)$.
For any $q\geq 1$, the Lie algebra cohomology }
\[ H^q(L^+(\fg,\sigma),  \hat{\mathscr{H}}_c(\lambda)\otimes V(\vec{\mu})_{\vec{z}})^{\fg^\sigma}   =0, \quad \text{ if }\, d_{\vec{\mu}}< \frac{ c+2\check{h}}{m},\]
{\it and the points $\vec{z}= \{z_0 = 0, z_1, \dots, z_s\}$ in the open unit disc are far apart in the hyperbolic metric. 
Here $\check{h}$ is the dual Coxeter number of $\fg$ and $d_{\vec{\mu}}$ is defined by the identity (\ref{d_mu}).}

\vskip1ex

The above  theorem,  in the special case $\sigma$ is the identity automorphism,  recovers Teleman's higher cohomology vanishing result (cf. \cite[Theorem 0 (c)]{Te}).

To prove this vanishing theorem, we generally follow the strategy of Teleman's proof in the untwisted case.  However, there is one major difference in that the M{\"o}bius transformations do not preserve the twisted affine Lie algebras (unlike the untwisted affine case). In the untwisted case, this allowed Teleman to reduce the problem of any marked point $z$ in $\mathbb{C}$ to the point $0$, which made some of the computations much easier to deal with. We establish a twisted analogue of Nakano identity in Theorems \ref{nakano} and \ref{newnakano}.  The proof of this identity requires a cumbersome computation of Laplacian operators, which is carried out in Section \ref{section3}. Some results on Casimir operators on graded Lie algebras in Section \ref{sect_casimir}, especially Proposition \ref{prop_casimir_id} is used crucially in Section \ref{section3}. We would like to stress that, in our above vanishing theorem,  the restriction on $\vec{\mu}$ is not optimal.

As a corollary of the above theorem, we get the following (cf. Corollary \ref{coro3.18}).
\vskip1ex

\noindent
{\bf Corollary.} 
 Under the assumptions of the above theorem, for any $q\geq 1$ we have
  \[ H_q(L^-(\fg,\sigma^{-1}),  \mathscr{H}_c(\lambda^*)\otimes V(\vec{\mu}^*)_{\vec{z}^{-1} } )^{\fg^\sigma}   =0, \quad \text{ if }\, d_{\vec{\mu}}< \frac{ c+2\check{h}}{m},\]
where $\vec{z}^{-1} =(\infty, z_1^{-1}, \cdots, z_s^{-1})$, 
$\mathscr{H}_c(\lambda^*)$ is the highest weight module of the twisted affine Lie algebra $\hat{L}(\fg,\sigma^{-1})$,
and $V(\vec{\mu}^*)_{\vec{z}^{-1} }:=V(\mu_0^*)_{\infty}\otimes \cdots  \otimes V(\mu_s^*)_{z_s^{-1}} $ represents the evaluation representation of $L^-(\fg,\sigma^{-1})$ in the usual sense.  
\vskip1ex

 The above corollary  implies the analogue of the Verlinde dimension formula for twisted conformal blocks on $\mathbb{P}^1$ with three marked points under the restriction on $\vec{\mu}$ and $\sigma$ as in the above corollary (cf. \cite[Theorem 6.9]{HK2}).  

\vskip3ex
\noindent
{\bf Acknowledgements:} We thank Constantin Teleman for many helpful correspondences and conversations. The first author was partially supported by the NSF grant DMS-2001365 and the second
author was partially supported by the NSF grant DMS-1802328.

\section{Preliminaries}
\subsection{Kac-Moody theory}
\label{Kac_sect}

Let $\fg$ be a simple Lie algebra over $\mathbb{C}$.  Let $\sigma$ be an automorphism of order $m$ of $\fg$.  Let $\mathcal{K}$ be the field of Laurent series in the parameter $t$,  such that $\sigma(t)=\epsilon^{-1}t$ where $\epsilon=e^{ \frac{2\pi \mathrm{i}}{m} }$ and $\sigma$ acts on $\mathbb{C}$ trivially.      Let $\mathcal{O}$ be the field of formal power series in $t$. We now define a central extension  $\hat{L}(\fg,\sigma):=\fg(\mathcal{K})^\sigma\oplus \mathbb{C}C$ of $\fg(\mathcal{K})^\sigma$ under the bracket
  \begin{equation}  \label{eq1.1.1.4}
[x[P]+z C, x'[P'] +z' C] =
[x,x' ][PP'] +m^{-1}\Res_{t=0} \,\bigl(({dP})
P'\bigr) \langle x,x'\rangle C,  
 \end{equation}
for  $x[P],x'[P']\in \fg(\mathcal{K})^\sigma$, $z, z'\in\bc$; where
$\Res_{t=0}$ denotes the coefficient of $t^{-1}dt$ and $\langle\,,\,\rangle$ denotes the normalized invariant form on $\fg$
so that  the induced form on $\fg^*$ satisfies $\langle\theta,\theta\rangle =2$ for the highest root $\theta$ of $\fg$.

{\it Throughout the paper, we fix a positive integer (called the level) $c>0$.  We also fix an integer $s>0$ denoting the number of marked points.}

We use $D_{c,\sigma}$ to denote the set of highest weights of $\fg^\sigma$ which parametrizes the integrable highest weight modules of $\hat{L}(\fg, \sigma)$  of level $c$, where the level denotes the action of $C$,  see \cite[Section 2]{HK1}. When $\sigma$ is trivial, we also use $D_c$ to denote this set for brevity.    For each $\lambda\in D_{c,\sigma}$, we will denote by $(\mathscr{H}_c(\lambda), \rho_\lambda)$ (or for simplicity $\mathscr{H}_c(\lambda)$) the associated integrable highest weight module of $\hat{L}(\fg, \sigma)$ of level $c$.

There exists a  `compatible' Cartan subalgebra $\fh$ and a `compatible' Borel subalgebra $\fb \supset \fh$ of $\fg$  both stable under the action of $\sigma$ such that 
\begin{equation}  \label{eq1.1.1.0} \sigma=\tau \epsilon^{{\rm ad} h}, 
\end{equation}
where $\tau$ is a (possibly trivial) diagram automorphism of $\fg$ of order $r$ preserving $\fh$ and $\fb$, $\alpha (h)\in \bz$ for any root $\alpha$ of $\fg$ and $\epsilon^{{\rm ad} h}$ is the inner automorphism of $\fg$ such that for any root $\alpha$ of $\fg$, $\epsilon^{{\rm ad} h}$ acts on the root space $\fg_\alpha$ by the multiplication $\epsilon^{\alpha(h)}$, and $\epsilon^{{\rm ad} h}$ acts on $\fh$ by the identity. Here  $h$ is an element in $\fh^\tau$. In particular,  $\tau$ and $\epsilon^{{\rm ad}h}$ commute.  Moreover,  $r$ divides $m$, $\alpha(h)  \in \mathbb{Z}^{\geq 0}$  for  any positive  root  $\alpha$  of  $\fg^\tau$ and $\theta_0(h)\leq \frac{m}{r}$ where $\theta_0\in (\fh^\tau)^*$ denotes the following weight of $\fg^\tau$:
\[ \theta_0=\begin{cases}  
\text{ highest root of } \fg,   \text{ if } r=1\\
\text{ highest short root  of  } \fg^\tau, \text{ if } r>1 \text{ and }(\fg, r)\neq (A_{2n},2)\\
  2\cdot \text{highest short root} \text{ of } \fg^\tau, \,\text{ if } (\fg,r)=(A_{2n},2).  \end{cases}  \]

Let  $\hat{L}(\fg, \tau)$ denote the Lie algebra with the construction similar to  $\hat{L}(\fg, \sigma)$ where $\sigma$ is replaced by $\tau$, $m$ is replaced by $r$ and $\epsilon$ is replaced by $\epsilon^{\frac{m}{r}}$. There exists an isomorphism of Lie algebras (cf. \cite[Theorem 8.5]{Ka}):
  \begin{equation} \label{neweqn2.3.1}  \Psi_\sigma:   \hat{ L}(\fg, \tau)\simeq   \hat{L}(\fg, \sigma)  
\end{equation}
given by $C\mapsto C$ and
 $x[t^j]\mapsto    x[t^{\frac{m}{r}j+k  } ]$, for any $x$ an  $\epsilon^{\frac{m}{r}j}$-eigenvector of $\tau$, and $x$ also a $k$-eigenvector of ${\rm ad } \,h$.  Then, the isomorphism $\Psi_\sigma$ induces a bijection 
 \begin{equation}
 \label{weight_formula}
   D_{c,\sigma}\simeq D_{c,\tau}, \quad  \lambda \mapsto \bar{\lambda} . \end{equation}

\begin{remark}
{\rm The explicit description of $D_{c, \sigma}$  is given in \cite[Lemma 2.1]{HK1} in terms of $\{n_{\lambda, i}\,|\,   i\in I(\fg^\tau)  \}$ defined there.  Also,  $\bar{\lambda}$ can be expressed in terms of numbers $a_i,a_i^\vee$ which can be read from \cite[p.54-55]{Ka} via \cite[Theorem 8.7]{Ka}. For the convenience of  readers, we would like to point out that there is  a typo in the formula for $\alpha_i^\vee$ in \cite[Theorem 8.7]{Ka}. The correct expression  is: $\alpha_i^\vee=1\otimes H_i + \frac{a_is_i r}{a_i^\vee m} K$. }
\end{remark}
\section{Vanishing theorem for Lie algebra cohomology }
\label{vanishing_Lie_sect}

\subsection{Contravariant Hermitian form on twisted affine algebras}\label{section2.1}
Let $\sigma$ be an automorphism of $\fg$ of order $m$. As in (\ref{eq1.1.1.0}), we write $\sigma=\tau \epsilon^{ad h}$ with respect to a compatible Borel subalgebra $\fb$ and a Cartan sublagebra $\fh$ contained in $\fb$, where $\epsilon=e^{\frac{2\pi \mathrm{i}}{ m }}$.  With respect to $\fb$ and $\fh$, we choose a set of Chevalley generators $\{ e_i,f_i,h_i \,|\, i\in I \}$ such that they are permuted by $\tau$.  

Let $\mathfrak{g}_{\mathbb{R}}$ be the real split form of $\fg$ generated by $\{ e_i,f_i,h_i \,|\, i\in I \}$. Let $\omega$ be the Cartan involution of $\fg$ defined by 
\[ \omega(e_i)=-f_i, \omega(f_i)=-e_i, \omega(h_i)=-h_i,  \quad \text{ for any }  i \in I  .\]
Define a {\it compact (conjugate-linear) involution} $\kappa$ of $\fg$:
\begin{equation}
\kappa=\omega\circ J ,\end{equation}
where $J$ is the complex conjugation of $\fg$ with respect to $\fg_\mathbb{R}$. 

\begin{lemma} 
\label{Lem_6.1}
The equality $\sigma\circ \kappa = \kappa\circ  \sigma$ holds on $\fg$.  
\end{lemma}
\begin{proof}
For each $i\in I$, we have 
\[  \sigma(e_i)=\epsilon^{\alpha_i(h)} e_{\tau(i)} ,     \sigma(f_i)=\epsilon^{-\alpha_i(h)} f_{\tau(i)} ,  \sigma(h_i)= h_{\tau(i)} ,   \]
where $e_i$ belongs to the root space $\fg_{\alpha_i}$. 
It is now easy to see that the equality $\sigma\circ \kappa = \kappa\circ  \sigma$ holds on $\{e_i,f_i,h_i\}$.   Thus, the equality holds on $\fg$.
\end{proof}

Consider the decomposition $\fg=\oplus_{ \underline{n}\in \mathbb{Z} /m\mathbb{Z}} \,  \fg_{ \underline{n}}$, where   $\fg_{ \underline{n}}$ is the $\epsilon^n$-eigenspace of $\sigma$ of $\fg$. Then, $\kappa(    \fg_{ \underline{n}})=  \fg_{ -\underline{n}}$ and $\fg_{\underline{0}}=\fg^\sigma$. 

As before, let $\langle  ,\rangle $ be the normalized invariant form on $\fg$ so that $\langle \theta, \theta\rangle = 2$ for the highest root $\theta$ of $\fg$. Then, $\langle , \rangle:    \fg_{ \underline{n}} \times   \fg_{ -\underline{n}}\to \mathbb{C}$ is non-degenerate for each $ \underline{n}$, and  the eigenspaces  $\fg_{ \underline{n}}$ and  $\fg_{- \underline{n}'}$ for $\underline{n} \neq \underline{n}'$  are orthogonal with respect to $\langle ,\rangle$. 

Define a Hermitian form $\{,\}$ on $\fg$ as follows:
\begin{equation}
  \{x,y\}=-\langle x, \kappa(y) \rangle, \quad  \text{ for any } x,y\in \fg .   \end{equation}
Then, $ \{, \}$ is a positive-definite contravariant Hermitian form on $\fg$, where contravariance means
\[\{[z,x], y\} = -\{x, [\kappa(z), y]\}, \,\,\,\text{for all $x, y, z\in \fg$}.\]
   Moreover,  $\fg=\oplus_{ \underline{n}\in \mathbb{Z} /m\mathbb{Z}} \,  \fg_{ \underline{n}}$ is an orthogonal decomposition with respect to $\{, \}$.  Thus, we can choose an orthonormal basis $\mathcal{B} =\mathcal{B}_\sigma $ with respect to $\{\,,\, \}$ such that $\mathcal{B}_{{\underline{n}}}:=\mathcal{B}\cap \fg_{ \underline{n} }$ is a basis of $ \fg_{ \underline{n} }$. We define a function $[\cdot ]:  \mathcal{B}\to \mathbb{Z}/m\mathbb{Z}$ given by $[b]=\underline{n}$, if $b \in \mathcal{B}_{{\underline{n}}}$. For each $b\in \mathcal{B}$,  set $\check{b}= -\kappa(b)$.  Then, $\{\check{b} \}_{b\in \mathcal{B} }$ is the dual basis of $\mathcal{B}$ with respect to $\langle , \rangle$. Moreover, $[\check{b}]=-\underline{n}$, if $[b]= \underline{n}$.

For $x\in \fg$ and $n\in \mathbb{Z}$, let $x(n) :=x[t^n]$.  The action of $\sigma$ on $\fg(\mathcal{K})$ is given in Section \ref{Kac_sect} and $\hat{L}(\fg, \sigma)$ defined there. Moreover, $\sigma$ acts trivially on $C$. 

Let $\hat{L}^f(\fg)$ denote the Lie subalgebra of $\hat{L}(\fg):= \hat{L}(\fg, I)$ obtained by replacing the function field $\mathcal{K}$ by the ring of Laurent polynomials $\mathbb{C}[t, t^{-1}]$. Similarly, 
$$\hat{L}^f(\fg, \sigma): =  \left(\fg \otimes_{\bc}\bc[t, t^{-1}]\right)^\sigma \oplus  \bc C.$$
Extend  the compact involution $\kappa$ of $\fg$ to $\hat{L}^f(\fg)$ as the conjugate-linear involution defined by 
\begin{equation}
  \kappa(x(n))=\kappa(x)(-n),   \quad  \text{ for any } x(n)\in \hat{L}^f(\fg),\,\,\,\text{and $\kappa(C) = -C$}.  \end{equation}
We take 
$$(\fg_{\bbR}\otimes \bbR[t^\pm]) \oplus \bbR C$$
as the split real form of $\hat{L}^f(\fg)$.

Similarly, we extend the contravariant Hermitian form $\{,\}$ on $\fg$ to $\hat{L}^f(\fg)$ by 
\begin{equation}   \{ x(n),  y(n')   \}= \{x,y\}\delta_{n, n'},\,\,\,\text{and $\{x(n), C\}=0$}.    \end{equation}
Similar to Lemma \ref{Lem_6.1},  $\sigma$ and $\kappa$ commute on $\hat{L}^f(\fg)$ since,  for any $x(n)\in \hat{L}^f(\fg) $,
\[  \sigma\kappa (x(n)) = \sigma ( \kappa (x) (-n) )=\epsilon^n (\sigma\kappa)(x)(-n) , \,\,\text{and $\sigma\kappa (C)=-C$},\]
and 
\[  \kappa\sigma(x(n)) =\kappa  ( \epsilon^{-n}  \sigma(x) (n) ) = \epsilon^n  (\kappa\sigma)(x)(-n) , \,\,\,\text{and $\kappa \sigma (C)= -C$}.  \]

 Thus,  $\kappa$ gives rise to a (conjugate-linear) {\it compact involution} of $\hat{L}^f(\fg,\sigma)$, and the induced Hermitian form $\{, \}$ on $\hat{L}^f(\fg,\sigma)$ is contravariant, and also positive-definite on $\fg[t,t^{-1}]^\sigma$ (cf.\, \cite[\S\S 7.6, 8.4]{Ka}).  Moreover,  $\kappa: L^\pm(\fg,\sigma)\to L^\mp(\fg,\sigma)$ is an anti-linear isomorphism,
 where  $  {L}^\pm(\fg, \sigma) := ( t^{\pm} \fg[t^\pm] )^\sigma$,  and  $  {L}^\pm(\fg, \sigma)$ are Lie subalgebras of $\hat{L}^f(\fg,\sigma)$. Moreover, the restriction of $\{, \}$ to $  {L}^\pm(\fg, \sigma)$ is positive-definite.  We now introduce  two different Hermitian forms $\{,\}_0$ and $\{,\}_1$ on $  {L}^-(\fg, \sigma)$ modified from $\{, \}$ by  defining 
 \begin{equation}
  \{   x(-n), y(-n') \}_0=  \frac{1}{n} \{ x(-n),  y(-n')  \} ,  \quad \text{ for any }  x(-n), y(-n')\in {L}^-(\fg, \sigma)  , \end{equation}
and 
 \begin{equation} \{   x(-n), y(-n') \}_1=  n  \{ x(-n),  y(-n')  \} ,  \quad \text{ for any }  x(-n), y(-n')\in {L}^-(\fg, \sigma). \end{equation}

We introduce the following set
\begin{equation} \hat{\mathcal{B}}_\sigma=\{ b(k) \,|\, k\in \mathbb{Z}_{<0} , b\in \mathcal{B} \,\text{with}\, [b]=\underline{k}  \} . \end{equation}
Then,
$ \hat{\mathcal{B}}_\sigma  $ is an orthonormal basis of $ L^-(\fg, \sigma) $ with respect to the Hermitian form $\{, \}$. Moreover, $ \{ \sqrt{-k} b(k)   \}_{ b(k) \in  \hat{\mathcal{B}}_\sigma}$ is an orthonormal basis of $L^-(\fg,\sigma)$ with respect to the Hermitian form  $\{,\}_0$, and $\{  \frac{1}{\sqrt{-k} }  b(k)   \}_{ b(k)\in  \hat{\mathcal{B}}_\sigma   }$ is an orthonormal basis of $L^-(\fg,\sigma)$ with respect to  $\{,\}_1$.

\begin{notation}\label{notation}
{\rm Fix a level $c>0$. Let $\lambda\in D_{c,\sigma}$, $\mu_0\in D_{c,\sigma^{-1}}$ and $\mu_1,\mu_2,\cdots, \mu_s\in D_c$. For $1\leq i\leq s$, let $V(\mu_i)$ be the irreducible representation of $\fg$ with highest weight $\mu_i$ and let $V(\mu_0)$ be the irreducible representation of $\fg^\sigma$ with highest weight $\mu_0$. 
Let $(\mathbb{P}^1, \vec{z} )$ be a (stable) $s+2$-pointed $\Gamma$-curve (for $s \geq 1$), where $\Gamma=\langle  \sigma\rangle$,  $\sigma$ acts on $\bbP^1$ via $\sigma z=e^{\frac{2 \pi i}{m}}z$ for $z\in \mathbb{P}^1$, $\vec{z}=(\infty, z_0,z_1,\cdots, z_s)$ with $z_0=0$ and $0< |z_i|< 1$ for other $i$. For $0\leq i\leq s$, let $V(\mu_i)_{z_i}$ be the evaluation representation of $\fg[t^{-1}]^\sigma$ with the underlying space $V(\mu_i)$, defined by
\[ x(-n)\cdot v= \bar{z}_i^{n} x\cdot v, \quad \text{for any } x(-n)\in \fg[t^{-1}]^\sigma ,\]
where $\bar{z}_i$ is the complex conjugate of $z_i$.  For each $0\leq i\leq s$, we define a Lie algebra  action of $L^+(\fg,\sigma)$ on $V(\mu_i)_{z_i}$ via 
\[  x(n)\cdot v= z_i^{n} x\cdot v ,  \quad \text{for any } x(n)\in L^+(\fg,\sigma).\]
We further define $C\cdot v =0.$ Observe that these actions {\it do not combine to give a Lie algebra action of $\hat{L}^f(\fg, \sigma)$.}

Set
\[ V(\vec{\mu})_{\vec{z}}=V(\mu_0)_{z_0}\otimes V(\mu_1)_{z_1}\otimes \cdots  \otimes  V(\mu_s)_{z_s}.\]  

We take the positive-definite contravariant Hermitian form on the $\hat{L}(\fg,\sigma)$-module $\mathscr{H}_c(\lambda)$ (cf. \cite[Theorem 2.3.13]{Kbook}), and the $\fg^{\Gamma_{z_i}}$-module  $V(\mu_i)$ (still denoted by $\{,\}$), where $\Gamma_{z_i}$ is the isotropy of $\Gamma$ at $z_i$. }
\end{notation}

Now, the following lemma is a restatement of $\{\,,\,\}$ being contravariant.
\begin{lemma}
\label{lem_hem1}
For any $v,w\in \mathscr{H}_c(\lambda)$ (resp. $V(\mu_i)_{z_i} $), we have 
\[  \{x(n)v, w \}=- \{ v, \kappa(x(n))\cdot w  \}  , \quad \text{ for any } x(n)\in L^f(\fg,\sigma). \]
\end{lemma}

From the Hermitian form $\{,\}$ on $\hat{L}^f(\fg,\sigma)$, one  introduces a Hermitian form on $\land^q \hat{L}^f(\fg,\sigma)$ in the standard way as follows. For any $\vec{y}=y_1\wedge \cdots \wedge y_q, \vec{u}=u_1\wedge \cdots \wedge u_q$ in $\wedge^q \hat{L}^f(\fg,\sigma)$, define 
\[  \{ \vec{y}, \vec{u}\}:= \det (\{ y_i, u_j\} ) . \]
Then, we have the follow lemma, which is easy to prove.
\begin{lemma}
\label{lem_herm2}
The Hermitian form $\{,\}$ on $\wedge^q \hat{L}^f(\fg,\sigma)$ with respect to the adjoint action of $\hat{L}^f(\fg,\sigma)$ is contravariant, i.e.,
\[  \{ ad_x( \vec{y}   ),  \vec{u}  \}=- \{ \vec{y},  ad_{ \kappa(x) } (\vec{u})   \},    \]
for any $x\in \hat{L}^f(\fg,\sigma)$, and $\vec{y}, \vec{u}\in  \wedge^q \hat{L}^f(\fg,\sigma)$.

Similarly,
$$\langle ad_x( \vec{y}   ),  \vec{u} \rangle = - \langle \vec{y},  ad_{x } (\vec{u})   \rangle.$$
\end{lemma}

Let $\bar{L}^-(\fg)$ (resp. $\bar{L}^-(\fg, \sigma)$ be the $L^2$-completion of ${L}^-(\fg)$ (resp. ${L}^-(\fg, \sigma)$)
with respect to the positive definite Hermitian form $\{ ,\}_0$ on ${L}^-(\fg)$ (resp. ${L}^-(\fg, \sigma)$).

\subsection{Chain complex}\label{chain}

Set $ M_q:= \wedge^q( L^-(\fg,\sigma))  \otimes \mathscr{H}_c(\lambda) \otimes V(\vec{\mu})_{\vec{z}} $  as a representation of $\fg[t^{-1}]^\sigma$.   We first consider  the standard chain complex $(C_*, \partial_*)$ of $L^-(\fg,\sigma)$ with coefficients in $M_q$ (cf. \cite[$\S$3.1.1]{Kbook}).   We calculate the formal adjoint operator $\partial_1^*$ of $\partial_1$ with respect to the Hermitian form $\{,\}_1$ on $L^-(\fg,\sigma)$, the Hermitian form $\{,\}_0$ on $\wedge^q {L}^-(\fg,\sigma)$ (tensor component of $M_q$), and the contravariant Hermitian forms $\{,\}$ on $\mathscr{H}_c(\lambda)$ and $V(\vec{\mu})_{\vec{z}}$.

Consider the standard differential  $\partial_p: \wedge^p (L^-(\fg, \sigma))\otimes  M_q\to  \wedge^{p-1} (L^-(\fg, \sigma))\otimes 
M_q$. In particular, when $p=1$, $\partial_1( x(-n )\otimes m ) = -x(-n)\cdot m $, for any $x(-n)\in  L^-(\fg, \sigma) $ and $m\in M_q$.    Take $\vec{x}(-\vec{p}):= x_1(-p_1)\wedge \cdots \wedge x_q(-p_q) \in \wedge^q L^-(\fg,\sigma)$, and $v\in \mathscr{H}_c(\lambda) \otimes V(\vec{\mu})_{\vec{z}}  $. Also, take an orthonormal basis $\{v_\phi\}_\phi$ of $\mathscr{H}_c(\lambda) \otimes V(\vec{\mu})_{\vec{z}} $. Then,  we get (where $ \hat{\mathcal{B}}_\sigma^q$ is an orthonormal basis of $\wedge^q L^-(\fg,\sigma)$  with respect to the Hermitian form $\{\,,\,\}$ obtained as wedge product of an orthonormal basis of $L^-(\fg, \sigma)$ and for $\vec{n}= (n_1, \dots, n_q), |\vec{n}|:=n_1\cdots n_q$)

\begin{align} \label{eqnnew121}
 \partial_1^* ( \vec{x}(-\vec{p})\otimes v ) &=\sum_{\vec{a}( -\vec{n}) \in \hat{\mathcal{B}}_\sigma^q, b(-k)\in \hat{\mathcal{B}}_\sigma, v_\phi   }  \frac{| \vec{n}| }{k}  \big \{  \partial_1^*(\vec{x}(-\vec{p})\otimes v  ) , b(-k) \boxtimes \vec{a} (-\vec{n})\otimes v_\phi  \big \}\,  b(-k) \boxtimes \vec{a}(-\vec{n})\otimes v_\phi   \notag \\
  &=-\sum_{\vec{a}( -\vec{n}) \in \hat{\mathcal{B}}_\sigma^q, b(-k)\in \hat{\mathcal{B}}_\sigma , v_\phi  }     \Big (  \frac{|\vec{n}|}{k} \big \{  \vec{x}(-\vec{p})\otimes v   ,    ad_{b(-k)} ( \vec{a} (-\vec{n})   )\otimes v_\phi  \big\}\,  b(-k) \boxtimes \vec{a}(-\vec{n})\otimes v_\phi  \notag\\
   &\quad \quad \quad\quad \quad   +   \frac{ |\vec{n} | }{k} \big \{  \vec{x}(-\vec{p})\otimes v   ,      \vec{a}(-\vec{n})   \otimes b(-k)\cdot  v_\phi  \big \} \, b(-k) \boxtimes \vec{a}(-\vec{n})\otimes v_\phi   \Big ) \notag\\
 &= -\sum_{\vec{a}( -\vec{n}) \in \hat{\mathcal{B}}_\sigma^q, b(-k)\in \hat{\mathcal{B}}_\sigma , v_\phi  }  \Big (  \frac{ |\vec{n}| }{k}  \big \{  \vec{x}(-\vec{p})  ,    ad_{b(-k)} (  \vec{a} (-\vec{n}) )  \big \}_0  \big \{ v, v_\phi  \big \} \, b(-k) \boxtimes \vec{a}(-\vec{n})\otimes v_\phi  \notag\\
 &\quad \quad \quad\quad \quad   +   \frac{ |\vec{n}| }{k} \big \{  \vec{x}(-\vec{p})  ,    \vec{a} (-\vec{n})  \big \}_0 \{ v,     b(-k)\cdot v_\phi  \big \} \,   b(-k) \boxtimes \vec{a}(-\vec{n})\otimes v_\phi      \Big ) \notag\\
&= -\sum_{\vec{a}( -\vec{n}) \in \hat{\mathcal{B}}_\sigma^q, b(-k)\in \hat{\mathcal{B}}_\sigma , v_\phi  }  \Big ( \frac{ |\vec{n}| }{k |\vec{p}|}   \big \{ ad_{\check{b}(k)} ( \vec{x}(-\vec{p}) ) ,     \vec{a} (-\vec{n})  \big \}  \big  \{ v, v_\phi  \big \} \, b(-k) \boxtimes \vec{a}(-\vec{n})\otimes v_\phi  \notag \\
 &  -  \frac{ |\vec{n}| }{k}  \big \{  \vec{x}(-\vec{p})  ,    \vec{a} (-\vec{n}) \big \}_0 \big \{ \check{b}(k)\cdot v,     v_\phi \big \}  \,  b(-k) \boxtimes \vec{a}(-\vec{n})\otimes v_\phi   \Big ), \quad \text{by Lemmas \ref{lem_hem1} and  \ref{lem_herm2} }\notag\\
 &=-\sum_{b(-k)\in \hat{\mathcal{B}}_\sigma } \Big ( \sum_{i=1}^q \left(\frac{p_i-k}{kp_i} b(-k) \boxtimes \overline{ad}^i_{\check{b}(k) } ( \vec{x}(-\vec{p}) )\otimes v\right)  +   \frac{b(-k)}{k}\boxtimes \vec{x}(-\vec{p})\otimes \check{b}(k)\cdot v \Big  ),
  \end{align}
where $ad^i_{\check{b}(k)} (  \vec{x}(-\vec{p}) )$ denotes the action only on the $i$-th factor, and $\overline{ad}$ denotes the projection onto $L^-(\fg,\sigma)$.  Observe that the second term of the above sum is an infinite sum taking values in $\bar{L}^-(\fg,\sigma)\otimes M_q$.

\vskip1ex
From the above equation \eqref{eqnnew121} and the definition of $\partial_1$, the following proposition follows easily.
\begin{proposition}\label{newprop6.5} Define $\Box=\partial_1\circ \partial_1^*$.  Then, as an operator from $M_q$ to $\bar{M}_q:= 
 \overline{\wedge^q({L}^-(\fg,\sigma))  \otimes \mathscr{H}_c(\lambda) \otimes V(\vec{\mu})_{\vec{z}}} $, 
\begin{align*}
 \Box(\vec{x}(-\vec{p})\otimes v) &=  \sum_{b(-k)\in \hat{\mathcal{B}}_\sigma }  \sum_{i=1}^q \left( \frac{p_i-k}{kp_i}  \Big ( ad_{ b(-k) }\big ( \overline{ad}^i_{\check{b}(k) } ( \vec{x}(-\vec{p}) ) \big)\otimes v +  \overline{ad}^i_{\check{b}(k) } ( \vec{x}(-\vec{p}) )\otimes b(-k)\cdot v\Big )\right)\\
& \quad + \sum_{ b(-k)\in \hat{\mathcal{B}}_\sigma} \frac{1}{k}  \Big ( ad_{b(-k)}( \vec{x}(-\vec{p})  ) \otimes  \check{b}(k)\cdot v +\vec{x}(-\vec{p})\otimes b(-k) \check{b}(k)\cdot v   \Big ),
\end{align*}
where $\bar{M}_q$ is the $L^2$-completion of $M_q$ with respect to the metric $\{\,,\,\}0$ on $L^-(\fg, \sigma)$ and the standard metrics on 
 $\mathscr{H}_c(\lambda)$ and  $V(\vec{\mu})_{\vec{z}}$ as in Subsection \ref{chain}.
  
Observe that the above sum makes sense as an element of $\bar{M}_q$ since each $|z_i|<1$ by assumption.
\end{proposition}

\begin{definition} \label{sect_special_aut}{\rm 
An automorphism $\sigma$ of $\fg$ is called $\mathbf{special}$ if $\sigma$ is a   diagram automorphism (including the identity automorphism), or an order 4 automorphism of $\fg$ or its inverse when $\fg$ is of type $A_{2n}$, which is defined as follows. Let $e_i,f_i,h_i, i=1,\cdots, 2n$,  be the set of Chevalley generators. The automorphism $\sigma$ of $\fg$ is defined such that 
\begin{equation}
\label{aut_def1}
\begin{cases}
\sigma(e_i)=e_{\tau(i)},   \quad  \text{ if } i\not= n, n+1;    \\
 \sigma(e_i)=\sqrt{-1} e_{\tau(i)},   \quad \text{ if } i\in \{n, n+1\} ; \\
  \sigma(f_\theta)=f_\theta ,
\end{cases},
\end{equation}
where $\theta$ is the highest root of $\fg$ and $\tau$ is the nontrivial diagram automorphism.  In fact, we can write 
\beq  
\label{as_auto}
 \sigma= \tau \sqrt{-1}^{{\rm ad }h },\eeq
 where $h\in \fh$ is such that $\alpha_i(h)=0$ if $i\not=n, n+1$ and $\alpha_i(h)=1$ if $i=n,n+1$. As mentioned above, $\sigma^{-1}$ is also a special automorphism.

We call a special automorphism $\sigma$ to be a {\it standard special} automorphism (or simply a {\it standard} automorphism) if $\sigma$ is special and is not a nontrivial diagram automorphism when $\fg$ is of type $A_{2n}$. (Observe that standard means that $\fg^\sigma$ is the same as $\mathring{\fg}$ as  defined in \cite{Ka}.)

The following table describes the fixed point Lie algebra for all the special automorphisms (cf.\,\cite[Section 2.1]{BH}):
 \begin{equation}
\label{Fix_table}
 \begin{tabular}{|c  | c | c |c |c |c | c| c| c|c |c|c|c|c|c |c ||} 
 \hline
$(\fg, m)$   & $(A_{2n-1}, 2 ) $  &  $(A_{2n}, 4)$  &   $(A_{2n}, 2)$  &  $(D_{n+1} , 2  ) $  &           $ (D_4,   3)$  &  $ (E_6,  2)$    \\ [0.8ex] 
 \hline
$ \fg^\sigma $  &  $ C_n $  &   $ C_n $   &    $B_n$ &  $B_n $   &    $G_2 $ &  $F_4 $  \\ [0.8ex] 
 \hline
\end{tabular},
\end{equation}
where by convention $C_1$ and $B_1$ are $A_1$ and $n\geq 3$ for $D_{n+1}$.  }
\end{definition}

\subsection{Casimir operators on graded Lie algebras}
\label{sect_casimir}

Recall that $\langle , \rangle$ is the normalized invariant form on $\fg$. It induces an invariant form on $\fg^\sigma$ via restriction, which is still denoted by $\langle , \rangle$.
\begin{lemma}
\label{pairing_cox_lem}Let $\sigma$ be a standard automorphism of $\fg$  of order $m$ and let $\theta_\sigma$ be the element defined below:

\beq
\label{theta_1}
 \theta_\sigma=\begin{cases}
\text{ highest short root of   } \fg^\sigma,    \quad  (\fg, m)\not= (A_{2n}, 4)\\
\frac{1}{2}\text{ highest  root of $\fg^\sigma$},    \quad  (\fg, m)= (A_{2n}, 4).

 \end{cases}   \eeq

 Then, $\langle \theta_\sigma, \theta_\sigma+2\rho_\sigma \rangle =\frac{2\check{h}}{m}$, where $\check{h}$ as earlier is the dual Coxeter number  of $\hat{L}(\fg, \sigma)$ (or $\fg$), and $\rho_\sigma$ is half the sum of positive roots of $\fg^\sigma$.
\end{lemma}
\begin{proof}
Let $(,)$ be the standard pairing between $\fh^\sigma$ and $(\fh^\sigma)^*$. Then, 
\[ \langle \theta_\sigma, \theta_\sigma+2\rho_\sigma \rangle =\frac{ ( 2\check{\theta}_\sigma,    \theta_\sigma+ 2 \rho_\sigma  )}{ \langle\check{\theta}_\sigma,  \check{ \theta}_\sigma   \rangle },\]
where $\check{\theta}_\sigma$ is the element defined below:

\beq
\label{theta_2}
\check{ \theta}_\sigma=\begin{cases}
\text{ highest coroot of   } \fg^\sigma,    \quad  (\fg, m)\not= (A_{2n}, 4)\\
2\cdot \text{ highest short coroot of  } \fg^\sigma ,   \quad  \quad  (\fg, m)= (A_{2n}, 4).
 \end{cases}   \eeq

 When $(\fg,m)$ is different from  $(A_{2n}, 4)$, as $\check{\theta}_\sigma$ is the highest coroot, $ \langle\check{\theta}_\sigma,  \check{ \theta}_\sigma   \rangle = 2m$.  When $(\fg,m)$ is $(A_{2n}, 4)$, from an explicit realization of $\fg^\sigma$, one can see that for any short coroot $\check{\beta}$, $\langle\check{\beta}, \check{\beta}\rangle=2$. Since $\check{\theta}_\sigma$ is twice the highest short coroot,   $ \langle\check{\theta}_\sigma,  \check{ \theta}_\sigma   \rangle = 8=2m$.
\end{proof}

Let $\sigma$ be a special automorphism of $\fg$ of order $m$. Let $\theta_{l}$  denote the highest root of $\fg^\sigma$ and let $\theta_s$ denote the highest short root of $\fg^\sigma$. Let $\theta_{\underline{i}}$ be the unique  maximal weight 
(with respect to the Bruhat-Chevalley order on the set of weights) of $\fg^\sigma$ that appears in $\fg_{\underline{i}}$. 
\begin{lemma}
\label{lem_hw_g_i}
 We have 
\[  \theta_{\underline{i}}=  \begin{cases}     
\theta_l  \quad   \text{ if }\, \underline{i}=\underline{0} \\
\theta_s  \quad   \text{ if }\, \fg\not= A_{2n} \text{ and } (m, \underline{i})= (2, \underline{1}) \text{ or }  (3, \underline{1}) \,\text{or} \, (3, \underline{2})\\ 
0\quad    \text{ if }\, \fg = A_{2} \, \text{and} 
   \,  (m, \underline{i})=(4, \underline{2}) \\
\theta_s\quad    \text{ if }\, \fg = A_{2n}  (n>1)\, \text{and} \,
     (m, \underline{i})=(4, \underline{2}) \\

\frac{\theta_l}{2}   \quad   \text{ if } 
\, \fg = A_{2n}  (n\geq 1)\, \text{and} \,
 (m, \underline{i})= (4, \underline{1}),  (m, \underline{i})= (4, \underline{3})  \\
2\theta_s   \quad  \text{ if } \, (\fg, m)= (A_{2n}, 2) \text{ and } \underline{i}=\underline{1}.
\end{cases} . \]

If $(\fg,m)\not= (A_{2n}, 4 )$ or $\underline{i}\not= \underline{2} $, $\fg_{\underline{i}}$ is an irreducible representation of $\fg^\sigma$ of highest weight $\theta_{\underline{i}}$. If  $(\fg,m)= (A_{2n}, 4 ) (n>1)$,   $\fg_{\underline{2}}\simeq V(\theta_s )\oplus  \mathbb{C}$. If  $(\fg,m)= (A_{2}, 4 )$,   $\fg_{\underline{2}}\simeq  \mathbb{C}$. 
\end{lemma}
\begin{proof}
This lemma follows from \cite[Proposition 8.3]{Ka} for $\sigma$ a diagram automorphism. For the standard automorphism $\sigma$ of order $4$ of $A_{2n}$, it follows from the calculation given in the proof of  Proposition \ref{prop_casimir_id}. 
\end{proof}

\begin{proposition}
\label{prop_casimir_id}
Let $\sigma$ be a special automorphism of order $m$ of $\fg$. Then, for any $x(-p)\in L^-(\fg,\sigma)$, 
\begin{equation}
\label{casimir_coxeter_lem}
 \sum_{k\geq 0, b\in \mathcal{B} \,\text{with}\, [b]=-\underline{k}  } ad_{b(-k)} \overline{ad}_{\check{b} (k)} \cdot x(-p) = \frac{2p\check{h}}{m} x(-p)  .\end{equation}
\end{proposition}
\begin{proof}
Observe that
\[ S:= \sum_{k\geq 0,  b\in \mathcal{B} \,\text{with}\,
[b]=-\underline{k}  } ad_{b(-k)} \overline{ad}_{\check{b} (k)} \cdot x(-p) =\sum_{k=0}^{p-1} C_{-\underline{k}} \cdot x(-p)   , \]
where we set 
\begin{equation} C_{-\underline{k}} :=\sum_{ b\in \mathcal{B} \,\text{with}\,
[b]=-\underline{k} } ad_{b}  ad_{\check{b}}. \end{equation}
Then, $C=\sum_{k=0}^{m-1}C_{-\underline{k}}$ is the Casimir operator of $\fg$ with respect to the normalized invariant form. Its action on any irreducible $\fg$-module of highest weight $\lambda$ is given by the scalar $\langle \lambda, \lambda+2\rho\rangle$ (cf. \cite[Lemma 2.1.16]{Kbook}).
\vskip1ex

{\bf Case 1:} $(\fg,m)=(A_{2n-1}, 2) (n\geq 2); (D_{n+1}, 2) (n\geq 3); (E_6, 2)$.  When $p=2i$ for $i\geq 1$,  
\[S= i C \cdot x(-p) = 2i \check{h} x(-p)=  \frac{2p\check{h}}{m} x(-p),\,\,\,\text{since $C$ acts on $\fg$ via multiplication by $\check{h}$} .\]
When $p=2i+1 (i\geq 0)$ , 
\[ S= \sum_{k=0}^{2i-1}C_{-\underline{k}}\cdot x(-p)+C_{\underline{0}} \cdot x(-p) =2i \check{h} x(-p)+C_{\underline{0}} \cdot x(-p) . \]
Since $x\in \fg_{\underline{1}}$ and $ \fg_{\underline{1}}$ is an irreducible representation of $\fg^\sigma$ with highest weight $\theta_\sigma$, by Lemma \ref{pairing_cox_lem},  the formula (\ref{casimir_coxeter_lem}) holds.
\vskip1ex

{\bf Case 2:} $(\fg, m)= (D_4, 3)$.  When $p=3i$ or $p=3i+1$, the proof is  similar to the  Case 1, where we use Lemma \ref{pairing_cox_lem} for $p=3i+1$.  When $p=3i+2$,
\[   S= \sum_{k=0}^{3i+2}C_{-\underline{k}}  \cdot x(-p)-C_{\underline{1}}\cdot x(-p)=\frac{2(p+1)}{3}\check{h}x(-p)  - C_{\underline{1}}\cdot x(-p)  ,  \]
where $x\in \fg_{\underline{1}}$. 
The operator $C_{\underline{1}}$ on $\fg_{\underline{1}}$ commutes with the action of $ \fg^\sigma$, as  $\langle , \rangle: \fg_{\underline{1}}\times \fg_{\underline{2}}\to \mathbb{C}$ is $\fg^\sigma$-invariant and $\{b\}_{b\in \mathcal{B}_{\underline{1}} }$ and $\{\check{b}\}_{b\in \mathcal{B}_{\underline{1}}}$ are dual bases with respect to $\langle, \rangle$. Since $\fg_{\underline{1}}$ is an irreducible representation of $\fg^\sigma$ with highest weight $\theta_\sigma$, by Schur lemma, $C_{\underline{1}}$ acts by a scalar $c_{\underline{1}}$ on $\fg_{\underline{1}}$. We  compute this number by choosing a basis $ \mathcal{B }_{\underline{1}}$ of $\fg_{\underline{1}}$ and a specific vector $v$ in $\fg_{\underline{1}}$. Following the Bourbaki labelling \, \cite [Planche IV]{Bo}, let $\{ \alpha_1,\alpha_2,\alpha_3,\alpha_4\}$ be the set of simple roots of $\fg$, and $\sigma$ be the cycle $(134)$ on the set of vertices $\{1,2,3,4\}$. 
Choose a Chevalley basis $\{e_\alpha, h_i\}$ of $\fg$. For each positive root $\alpha$ we denote the associated $sl_2$-triple by $\{ e_\alpha, e_{-\alpha}, h_\alpha\}$ satisfying $\sigma(e_{\pm \alpha})= e_{\pm \sigma(\alpha)}$, and $\langle e_\alpha, e_{-\alpha}\rangle=1$ (cf.\,\cite[\S 7.9]{Ka}).   We choose $ \mathcal{B }_{\underline{1}}$ to be the set of the following elements:
\begin{equation}
\label{basis_D_4}
 h_{\alpha_1}+\epsilon^2 h_{\alpha_3}+ \epsilon h_{\alpha_4} ,   e_{\pm \alpha_1}+\epsilon^2 e_{\pm \alpha_3}+ \epsilon e_{\pm \alpha_4}, e_{\pm(\alpha_1+\alpha_2)}+\epsilon^2 e_{\pm (\alpha_3 + \alpha_2)}+ \epsilon e_{\pm (\alpha_4+\alpha_2)} , \end{equation}
\[e_{\pm(\alpha_1+\alpha_2+\alpha_3)}+\epsilon^2 e_{\pm (\alpha_3 + \alpha_2+\alpha_4)}+ \epsilon e_{\pm (\alpha_4+\alpha_2+\alpha_1)} , \]
where $\epsilon=e^{\frac{2\pi \mathrm{i}}{3}   }$.
Then, the dual basis in $\fg_{\underline{2}}  $ is given by 
\begin{equation}
\label{basis_D_4_dual}
\frac{1}{6} (h_{\alpha_1}+\epsilon h_{\alpha_3}+ \epsilon^2 h_{\alpha_4} ), \frac{1}{3}  (e_{\mp \alpha_1}+\epsilon e_{\mp \alpha_3}+ \epsilon^2 e_{\mp \alpha_4}), \frac{1}{3}(e_{\mp(\alpha_1+\alpha_2)}+\epsilon e_{\mp (\alpha_3+ \alpha_2)}+ \epsilon^2 e_{\mp (\alpha_4+\alpha_2)}) ,\end{equation}
\[\frac{1}{3}(e_{\mp(\alpha_1+\alpha_2+\alpha_3)}+\epsilon e_{\mp (\alpha_3 + \alpha_2+\alpha_4)}+ \epsilon^2 e_{\mp (\alpha_4+\alpha_2+\alpha_1)}) .  \]
We choose the vector $v=h_{\alpha_1}+\epsilon^2 h_{\alpha_3}+ \epsilon h_{\alpha_4}$ in $\fg_{\underline{1}}$. Then,
 $C_{\underline{1}} v=\sum_{[b]=\underline{1}} ad_{b}ad_{\check{b}}v= 4 v$. In fact, for the first basis element $ b$ in (\ref{basis_D_4}), $ad_bad_{\check{b}}v=0$, and for any other $b$ there, $ad_bad_{\check{b}}v=\frac{2}{3}v$.  
Thus, $c_{\underline{1}}=4=\frac{2\check{h}}{3}$, as the dual Coxeter number of $D_4$ is 6. Thus, (\ref{casimir_coxeter_lem}) holds. 
\vskip1ex

{\bf Case 3:}  $(\fg, m)=(A_{2n}, 4  )$.  Let $\{e_{i,j}\}_{1\leq i\not= j\leq 2n+1}  \sqcup \{  e_{i,i}-e_{i+1,i+1} \}_{1\leq i \leq 2n} $ be the standard basis of  $\fg=sl_{2n+1}$, where $e_{i,j}$ is the $(2n+1)\times (2n+1)$-matrix with $(i,j)$-entry equal to  1 and zero elsewhere. 
Then,  $\fg^\sigma\simeq sp_{2n}$, and as representations of $sp_{2n}$, $\fg_{\underline{1}}$ and $\fg_{\underline{3}}$ are both isomorphic to the standard representation, which is of highest weight $\theta_\sigma$  (half of the highest root of $\fg^\sigma$) (cf.\,\cite[Planche III]{Bo}). 
When $p=4i$ or $p=4i+1$, the proof is similar to Case 1, where Lemma \ref{pairing_cox_lem} is used for $p=4i+1$.  

When $p=4i+3$,  the operator $C_{\underline{1}}$ acts on $\fg_{\underline{1}}$ by a scalar $c_{\underline{1}}$ as $\fg_{\underline{1}}$ is an irreducible $\fg^\sigma$-representation and   $C_{\underline{1}}$ commutes with the action of $\fg^\sigma$. 
 We choose a basis $\mathcal{B}_{\underline{1}}$ of $\fg_{\underline{1}}$ as follows:
\begin{equation} 
\label{basis_A_2n}
 \{   e_{i, n+1 }+(-1)^{n-i} e_{n+1, 2n+2-i}   ,  e_{n+1,i }-(-1)^{n-i} e_{2n+2-i, n+1}        \}_{1\leq i\leq n} .   \end{equation}
The corresponding dual basis of $\fg_{\underline{3}}$ is given by: 
\begin{equation}
\label{basis_A_2n_dual}
    \{  \frac{1}{2}(e_{n+1,i  }+(-1)^{n-i} e_{2n+2-i, n+1} )   , \frac{1}{2} ( e_{i, n+1 }-(-1)^{n-i} e_{n+1, 2n+2-i}  )    \}_{1\leq  i \leq  n} .    \end{equation}
 We choose the vector $v=e_{1, n+1}+ (-1)^{n-1} e_{n+1, 2n+1}$ of $\fg_{\underline{1}}$.  Then, when $b=v$, $ad_b ad_{ \check{b}}v= v$, and for any other $b\in\mathcal{B}_{\underline{1}} $, $ad_b ad_{ \check{b}}v= \frac{1}{2}v$. Thus, the scalar $c_{\underline{1}}= \frac{2n+1}{2}=\frac{2\check{h}}{m}$. 

When $p=4i+2$, it suffices to show that the operator $C_{\underline{0}}+C_{\underline{-1}}$ acts on $\fg_{\underline{2}}$ by the scalar $2n+1$. 

 When $n>1$, $\fg_{\underline{2}}$ is not irreducible as a representation of $\fg^\sigma$; in fact,  we have $\fg_{\underline{2}}=  W \oplus      \mathbb{C} v_0 $,  where $W$ is an irreducible representation of $\fg^\sigma$ with highest weight  the highest short root $\theta_s$ of $\fg^\sigma$, and 
\[  v_0:=\sum_{i=1}^n  (e_{i,i} +  e_{2n+2-i, 2n+2-i} ) -2n e_{n+1,n+1}  \]
 is an $\fg^\sigma$-invariant vector.

 We  show that $C_{\underline{0}}+C_{-\underline{1}}$ acts on $W$ and $v_0$ by the same scalar $2n+1$. First  notice that $C_{\underline{0}}$ and $C_{-\underline{1}}$ act on $W$ by scalars. Of course, 
   $C_{\underline{0}}$ acts on $W$ by the scalar
 \[  \langle \theta_s , \theta_s+2\rho_\sigma  \rangle= \frac{2(\check{\theta}_s, \theta_s+2\rho_\sigma   )}{\langle \check{\theta}_s, \check{\theta}_s\rangle } =1+ (\check{\theta}_s, \rho_\sigma) =2n. \]
 Further, $C_{-\underline{1}}$ acts on $W$ by a scalar since the action of  $C_{-\underline{1}}$ on $W$ commutes with the $\fg^\sigma$-action.  Moreover, one computes that $C_{-\underline{1}}$ acts on $W$ by 1 by using  the basis given in (\ref{basis_A_2n_dual}) and its dual basis in $(\ref{basis_A_2n})$ and fixing the vector $e_{1,2}-e_{2n, 2n+1}\in W$.   
 We now calculate the action of $C_{-\underline{1}}$ on $v_0$.  It turns out that it acts by the scalar $2n+1$. Furthermore, $C_{\underline{0}}$ clearly acts on $v_0$ by $0$. Thus, $C_{\underline{0}}+C_{-\underline{1}}$  acts on $\fg_{\underline{2}}$  by $2n+1$. 
 
 But, when $n=1$, $\fg_{\underline{2}}$ is the trivial one dimensional representation of $\fg^\sigma$ spanned by $v_0$.
 Hence,  $C_{\underline{0}}$ acts by zero on  $\fg_{\underline{2}}$. Moreover, $C_{-\underline{1}}$  acts on $v_0$ by $3=2n+1$. Thus,  in the case $n=1$  as well,  $C_{\underline{0}}+C_{-\underline{1}}$  acts on $\fg_{\underline{2}}$  by the scalar $2n+1$.

 \vskip1ex

{\bf Case 4:}  $(\fg, m)=(A_{2n}, 2)$. In this case, $\sigma$ is a diagram automorphism and $\fg^\sigma$ is of type $B_n$. Similar to Case 1, we need to compute the action of $C_{\underline{0}}$ on $\fg_{\underline{1}}$.  Note that $\fg_{\underline{1}}$ is an irreducible representation of $\fg^\sigma$ with highest weight $2\theta_s$, where $\theta_s$ is the highest short root of $\fg^\sigma$. Thus,  $C_{\underline{0}}$ acts on $\fg_{\underline{1}}$ by the scalar
\[ \langle 2\theta_s,  2\theta_s+ 2\rho_\sigma  \rangle=\frac{ 4 (\check{\theta}_s, 2\theta_s+2\rho_\sigma  )  }{  \langle \check{\theta}_s, \check{\theta}_s  \rangle} = \frac{ (\check{\theta}_s, 2\theta_s+2\rho_\sigma  )  }{ 2} =2+ (\check{\theta}_s, \rho_\sigma)=  2n+1= \frac{2\check{h}}{m},\]
where $ \langle \check{\theta}_s, \check{\theta}_s  \rangle=8$, since $\check{\theta}_s$ is conjugate to $2(\check{\alpha}_n+\check{\alpha}_{n+1} )$ (both are long coroots of $\fg^\sigma$) (cf.\,\cite[Case 5, p.129]{Ka}). Thus, the formula (\ref{casimir_coxeter_lem}) holds. 

\end{proof}

\begin{lemma}
\label{lem_eg_estimate}
Let $V(\mu)$ be an irreducible $\fg^\sigma$-module with highest weight $\mu$, where $\sigma$ is a special automorphism of $\fg$ of order $m$. Consider the operator 
\[ C_\mu^{ \underline{i}}:  \fg_{\underline{i}}\otimes V(\mu) \to  \fg_{\underline{i}}\otimes V(\mu) \]
defined by $ C_\mu^{ \underline{i}}(x\otimes v)= \sum_{b \in \mathcal{B}_{\underline{0}}}   [b, x]\otimes  \check{b}\cdot v $, where $\mathcal{B}_{\underline{0}} := \{b\in \mathcal{B}: [b]= \underline{0}\}$. 
Then, except for $(\fg = D_4, 3)$ and $ \underline{i}=  \underline{0}$, the eigenvalues $e^{\underline{i}}_\mu$ of $C_\mu^{ \underline{i}}$ are bounded by 
\begin{equation}
\label{ev_inequ}
  -\langle \mu +2\rho_\sigma,\theta_{\underline{i}} \rangle  \leq e^{\underline{i}}_\mu\leq   \langle \mu,\theta_{\underline{i}} \rangle .  \end{equation}

 \end{lemma}
\begin{proof}
Let $C_{\fg^\sigma}$ be the Casimir operator of $\fg^\sigma$ with respect to the invariant form induced from the normalized invariant form  $\langle ,\rangle$ on $\fg$. Then, for $x\in \fg_{\underline{i}} (\theta_{\underline{i}})$
(where $\fg_{\underline{i}} (\theta_{\underline{i}})$ denotes the $\theta_{\underline{i}}$-isotypic $\fg^\sigma$-module component of $\fg_{\underline{i}}$) and $v\in V(\mu)$, 
\begin{align*}
  -2C_\mu^{ \underline{i}} (x\otimes v)  &= - C_{\fg^\sigma} (x\otimes v)+ (  C_{\fg^\sigma} x)\otimes v  + x\otimes  C_{\fg^\sigma}v  \\
                     &= - C_{\fg^\sigma} (x\otimes v)+ \langle \theta_{\underline{i}}, \theta_{\underline{i}}+ 2\rho_\sigma  \rangle x\otimes v+ \langle \mu, \mu+2\rho_\sigma \rangle  x\otimes v.
\end{align*}
Now, any irreducible $\fg^\sigma$-submodule of $\fg_{\underline{i}}(\theta_{\underline{i}})
\otimes  V(\mu)$ is of the form $V(\mu+\gamma)$, for some weight $\gamma$ of $\fg_{\underline{i}}(\theta_{\underline{i}})$ such that $\mu+\gamma$ is $\fg^\sigma$-dominant.  The eigenvalue of $C_{\fg^\sigma}$ on $V(\mu+\gamma)$ is given by $\langle   \mu+\gamma, \mu+\gamma+2\rho_\sigma  \rangle$. Now, for any such $V(\mu+\gamma)$ , 
\begin{equation}
\langle   \mu-\theta_{\underline{i}}, \mu-\theta_{\underline{i}}+2\rho_\sigma  \rangle \leq  \langle   \mu+\gamma, \mu+\gamma+2\rho_\sigma  \rangle  \leq \langle   \mu+\theta_{\underline{i}}, \mu+\theta_{\underline{i}}+2\rho_\sigma  \rangle ,
   \end{equation}
where the second inequality can be easily seen since $\mu+\gamma$ is $\fg^\sigma$-dominant, and $\theta_{\underline{i}}-\gamma$ is a non-negative linear combination of the positive roots of $\fg^\sigma$. To prove the first inequality, by Lemma \ref{lem_hw_g_i},  we  observe that $\fg_{\underline{i}}(\theta_{\underline{i}})$ is self-dual, and thus $\gamma$ can be written as $-\theta_{\underline{i}}+\beta$, where $\beta$ is a non-negative linear comination of positive roots of $\fg^\sigma$. Moreover, $2\rho_\sigma-\theta_{\underline{i}}$ is a dominant weight, since $(\theta_{\underline{i}  }, \check{\beta} )\leq 2$  for any simple coroot $\check{\beta}$ of $\fg^\sigma$ except when 
$(\fg = A_2, 2)$ and $ \underline{i}=  \underline{1}$.  So assume that $(\fg, m)\neq (A_2, 2)$ when $ \underline{i}=  \underline{1}$.
Thus, any eigenvalue $e^{\underline{i}}_\mu$ of $C^{\underline{i}}_\mu$ acting on  $\fg_{\underline{i}}\otimes  V(\mu)$ has lower bound  
\[   
\frac{1}{2}\left(   \langle   \mu-\theta_{\underline{i}}, \mu-\theta_{\underline{i}}+2\rho_\sigma  \rangle- 
\langle \theta_{\underline{i}}, \theta_{\underline{i}}+ 2\rho_\sigma  \rangle-\langle   \mu, \mu+2\rho_\sigma   \rangle  \right)
\]
and upper bound
\[   
\frac{1}{2}\left(   \langle   \mu+\theta_{\underline{i}}, \mu+\theta_{\underline{i}}+2\rho_\sigma  \rangle - 
\langle \theta_{\underline{i}}, \theta_{\underline{i}}+ 2\rho_\sigma  \rangle-\langle   \mu, \mu+2\rho_\sigma \rangle \right),
\]
i.e.,  the inequality (\ref{ev_inequ}) holds. 

If we take $x\in \fg_{\underline{i}}(0)$ (which exists only for $(\fg, m)= (A_{2n}, 4)$ and $\underline{i} =\underline{2}$),
then $C_\mu^{ \underline{i}} (x\otimes v) = 0.$ Thus, the inequality \eqref{ev_inequ} is clearly satisfied.

For $(\fg = A_2, 2)$ and $ \underline{i}=  \underline{1}$, the equation \eqref{ev_inequ} is still satisfied. To prove this, if $\mu=n \omega, n\geq 2$ ($\omega$ being the fundamental weight of $\fg^\sigma =A_1$), the equation \eqref{ev_inequ} is satisfied by the same proof. If $\mu=0$, then $e_\mu^{\underline{1}}=0$ and if $\mu=\omega$, then $e_\mu^{\underline{1}}=-2, 3$. Thus,  the equation \eqref{ev_inequ} is satisfied in this case as well.

\end{proof}

\subsection{Cochain complex}\label{cochain}
We now consider the standard cochain complex $C^*:=C^*(L^+(\fg,\sigma), \hat{\mathscr{H}}_c(\lambda) \otimes V(\vec{\mu})_{\vec{z}} )$ of the Lie algebra $L^+(\fg,\sigma)$ with coefficients in  $\hat{\mathscr{H}}_c(\lambda) \otimes V(\vec{\mu})_{\vec{z}}$, where $\hat{\mathscr{H}}_c(\lambda)$ is the formal completion of $\mathscr{H}_c(\lambda)$ with respect to the energy grading (i.e., it is the direct {\it product} of the energy eigenspaces of $\mathscr{H}_c(\lambda)$), 
\[  C^q=\Hom_\mathbb{C}(\wedge^q L^+(\fg,\sigma), \hat{\mathscr{H}}_c(\lambda) \otimes V(\vec{\mu})_{\vec{z}} ),\]
and we denote $d^q: C^q\to C^{q+1}$ the standard differential in $C^*$ (cf. \cite[$\S$3.1.2]{Kbook}). Observe that 
\begin{equation} \label{neweqn2.5.1} \hat{\mathscr{H}}_c(\lambda) \simeq \left(\mathscr{H}_{-c}^{\low}(\lambda^{*_\sigma})\right)^*\,\,\text{the full vector space dual},
\end{equation}
where $\lambda^{*_\sigma}$ is the highest weight of the dual $\fg^\sigma$-module $V(\lambda)^*$
and $\mathscr{H}_{-c}^{\low}(\lambda^{*_\sigma})$ is the integrable {\it lowest weight} $\hat{L}^f(\fg, \sigma)$-module with
lowest energy $\fg^\sigma$-module the dual of the $\fg^\sigma$-module $V(\lambda)$. 

We take the Hermitian form $\{,\}_0$ on $L^-(\fg,\sigma)$ as defined before, and the $\hat{L}(\fg,\sigma)$ (resp. $\fg^{\Gamma_{z_i}}$)-contravariant positive-definite Hermitian form $\{,\}$ on $\mathscr{H}_c(\lambda)$ (resp. $V({\mu_i})_{z_i}$). 

 Define the invariant form $\langle\,,\,\rangle$ on $\hat{L}(\fg, \sigma)$ by
 \[\langle x[P], y[Q]\rangle = \langle x, y\rangle \Res_{t=0}(t^{-1}PQ),\,\,\langle x[P], C\rangle =\langle C, C\rangle=0,\,\,\,\text{for $x[P], y[Q]\in \fg(\mathcal{K})^\sigma$}.\]
 
The invariant form $\langle , \rangle$ on $\hat{L}(\fg,\sigma)$ gives rise to an embedding
 $\xi: L^-(\fg,\sigma)\to L^+(\fg,\sigma)^*$  given by   $\xi(X)(Y)=\langle X,Y \rangle$, for $X\in L^-(\fg,\sigma)  $ and $Y\in L^+(\fg,\sigma)$.  This induces an embedding 
\[  \wedge^q (L^-(\fg,\sigma))\otimes \mathscr{H}_c(\lambda) \otimes V(\vec{\mu})_{\vec{z}}  \longrightarrow  \Hom_\mathbb{C}(\wedge^q L^+(\fg,\sigma), \hat{\mathscr{H}}_c(\lambda) \otimes V(\vec{\mu})_{\vec{z}} )   . \]

Let $(d^{q-1})^*:  \wedge^q (L^-(\fg,\sigma))\otimes \mathscr{H}_c(\lambda) \otimes V(\vec{\mu})_{\vec{z}} 
 \to C^{q-1}$ be its formal  adjoint. By Lemma  \ref{lem7.1}, \\
 $d^q\left(\wedge^q (L^-(\fg,\sigma))\otimes \mathscr{H}_c(\lambda) \otimes V(\vec{\mu})_{\vec{z}} \right)$ lies in $  \wedge^{q+1} (\bar{L}^-(\fg,\sigma))\otimes \mathscr{H}_c(\lambda) \otimes V(\vec{\mu})_{\vec{z}}$ and by Lemma \ref{lem7.4}, \\ $(d^{q-1})^*\left(\wedge^q (L^-(\fg,\sigma))\otimes \mathscr{H}_c(\lambda) \otimes V(\vec{\mu})_{\vec{z}}\right)$ lies in $  \wedge^{q-1} (\bar{L}^-(\fg,\sigma))\otimes \mathscr{H}_c(\lambda) \otimes V(\vec{\mu})_{\vec{z}}$.
 Let  $\overline{\Square} := dd^*+d^*d$ be the Laplacian. We prove the following theorem giving an expression for the difference $\overline{\Square}-\Square$ restricted to 
$ \left[ \wedge^q\left(L^-(\mathfrak{g},\sigma)\right)\otimes \mathscr{H}_c(\lambda) \otimes V_{\vec{z}}(\vec{\mu})\right]^{\mathfrak{g}_{\underline{0}}}.$

\begin{theorem}\label{nakano}(Nakano identity)  Let $\sigma$ be a special automorphism  of $\fg$ of order $m$.
For $\vec{x}(-\vec{p})\otimes v_1\otimes v_2 \in \left[ \wedge^q\left(L^-(\mathfrak{g},\sigma)\right)\otimes \mathscr{H}_c(\lambda) \otimes V_{\vec{z}}(\vec{\mu})\right]^{\mathfrak{g}^\sigma},$ where $\vec{x}(-\vec{p}) = x_1(-p_1) \wedge \dots \wedge x_q(-p_q)$,
\begin{align*}
\left(\overline{\Square}-\Square\right)& \left( \vec{x}(-\vec{p})\otimes v_1\otimes v_2\right) =
 \frac{(c+2\check{h})q}{m}\vec{x} (-\vec{p})\otimes v_1\otimes v_2
+ \sum_{1\leq \ell\leq q; b(-k)\in \hB_\sigma} \frac{1}{p_\ell} ad^\ell_{b(-k)} \vec{x}(-\vec{p})\otimes v_1
\otimes  \left((1-(z\bar{z})^{p_\ell})z^k \check{b}\right) \cdot v_2\\
&+ \sum_{\ell; b(-k)\in \hB_\sigma}
 \frac{1}{p_\ell} \overline{ad}^\ell_{\check{b}(k)}\vec{x}(-\vec{p})\otimes v_1
\otimes  \left((1-(z\bar{z})^{p_\ell-k})
(\bar{z})^{k} b\right) \cdot v_2
+ \sum_{\ell; b_0\in \CB_{\underline{0}}} \frac{1}{p_\ell} ad^\ell_{b_0} \vec{x}(-\vec{p}) \otimes v_1 \otimes \left( (1-(z\bar{z})^{p_\ell}) \check{b}_0\right) \cdot v_2,
\end{align*}
where
$$ (z^{k_{1}}\bar{z}^{k_2}\check{b})\cdot v_2:= \sum_{i=0}^s\, v_2^0\otimes \dots \otimes z_i^{k_1}\bar{z}_i^{k_2}\check{b}\cdot v_2^i\otimes \dots \otimes v_2^s, \,\,\text{for $v_2= v_2^0\otimes \dots \otimes v_2^s$ with $v_2^i\in V(\mu_i)$}.
$$
Here $c$ is the scalar by which $C$ of $\hat{L}(\mathfrak{g}, \sigma)$ acts on $\mathscr{H}_c(\lambda)$ and $\check{h}$ is the dual Coxeter number of $\mathfrak{g}$. Observe that, from the above expression, we see that 
$$\left(\overline{\Square}-\Square\right)\left(\left[ \wedge^q\left(L^-(\mathfrak{g},\sigma)\right)\otimes \mathscr{H}_c(\lambda) \otimes V_{\vec{z}}(\vec{\mu})\right]^{\mathfrak{g}^\sigma}\right) \subset \left[ \wedge^q\left(\bar{L}^-(\mathfrak{g},\sigma)\right)\otimes \mathscr{H}_c(\lambda) \otimes V_{\vec{z}}(\vec{\mu})\right]^{\mathfrak{g}^\sigma}.$$
\end{theorem}
Its proof is given in Section \ref{section3}.

\subsection{Difference of Laplacians}
Recall that  $\bar{L}^-(\fg)$ (resp. $\bar{L}^-(\fg, \sigma)$) is  the $L^2$-completion of ${L}^-(\fg)$ (resp. ${L}^-(\fg, \sigma)$)
with respect to the positive definite Hermitian form $\{ ,\}_0$ on ${L}^-(\fg)$ (resp. ${L}^-(\fg, \sigma)$).

 Define the operator $T_z:   L^-(\fg,\sigma)\otimes V( \mu)_z \to  \bar{L}^-(\fg,\sigma) \otimes V( \mu)_z $ for any $z\in \mathbb{D} :=\{z\in \bc, |z|<1\}$,
and $V(\mu)$ an irreducible $\fg$-module with highest weight $\mu$ if $z\neq 0$ and $V(\mu)$ an irreducible $\fg^\sigma$-module if $z=0$, by
\begin{equation} 
\label{operator_T_z}
   T_z(x(-p) \otimes v )=  \begin{cases}  
 \sum_{b(-k)\in \hB_\sigma} \frac{1}{p} \left(ad_{b(-k)} x(-p)
\otimes   (1-(z\bar{z})^{p})
z^k \check{b} \cdot v
+ \overline{ad}_{\check{b}(k)} x(-p)
\otimes  (1-(z\bar{z})^{p-k})
(\bar{z})^{k} b \cdot v\right)\\
+ \sum_{b_0\in \CB_{\underline{0}}} \frac{1}{p} ad_{b_0} x(-p)  \otimes \left( 1-(z\bar{z})^{p}\right) \check{b}_0 \cdot v,  \quad   \text{ if } z\not= 0, \\
\sum_{b_0\in \CB_{\underline{0}}}  \frac{ 1}{p}
 [b_0, x]  (-p)    \otimes  \check{b}_0\cdot v ,   \quad \text{ if } z=0.
\end{cases}      . \end{equation}
Then, the operator $T_z$ extends to a multilinear operator
\begin{equation}
\label{op_T_z}
  T_{\vec{z}}:   \wedge^q (L^-(\fg,\sigma)) \otimes \mathscr{H}_c(\lambda)  \otimes V(\vec\mu)_{\vec{z}} \to  \wedge^q (\bar{L}^-(\fg,\sigma))  \otimes \mathscr{H}_c(\lambda)  \otimes V(\vec\mu)_{\vec{z}}   \end{equation}
defined by 
\[  T_{\vec{z}}\left(( x_1(-p_1)\wedge \cdots\wedge x_q(-p_q))
\otimes h\otimes v_0\otimes \cdots \otimes v_s  \right)=  \]
\[ \sum_{k=0}^s \sum_{\ell =1}^q   x_1(-p_1)\wedge \cdots \wedge T_{z_k} (  x_\ell(-p_\ell)\otimes v_k   )\wedge \cdots \wedge x_q(-p_q) \otimes h \otimes v_0\otimes\cdots   \otimes \hat{v_k} \otimes \cdots \otimes  v_s  , \]
for any $h\otimes v_0\otimes\cdots \otimes v_s \in   \mathscr{H}_c(\lambda)  \otimes V(\vec\mu)_{\vec{z}} $.
In view of the operator $  T_{\vec{z}}$, the 
 Nakano's identity (Theorem \ref{nakano}) clearly  takes the following form:

\begin{theorem} \label{newnakano} Follow the notation and assumptions from Theorem \ref{nakano}. In particular, $\sigma$ is a special automorphism of $\fg$ of order $m$. 
Then, for $\vec{x}(-\vec{p}) \otimes v \in [\wedge^q (L^-(\fg,\sigma)) \otimes \mathscr{H}_c(\lambda)\otimes V(\vec{\mu})_{\vec{z}}]^{\fg^\sigma}
$, where $\vec{x}(-\vec{p}):=  x_1(-p_1)\wedge \cdots\wedge x_q(-p_q)\in  \wedge^q L^-(\fg,\sigma)$ and $v\in 
 \mathscr{H}_c(\lambda)  \otimes V(\vec\mu)_{\vec{z}} $, 
we have  
\[ (\bar{\Box}- \Box )(  \vec{x}(-\vec{p}) \otimes v )= q\frac{c+2\check{h}}{m} \vec{x}(-\vec{p}) \otimes v +T_{\vec{z}}(\vec{x}(-\vec{p}) \otimes v) . \]
\end{theorem}

Let $V(\mu)$ be an irreducible representation of $\fg^\sigma$.  Now, consider the self-adjoint operator $S_0: \fg\otimes V(\mu)\to \fg\otimes V(\mu)$ given by 
\[ S_0 (x \otimes v  )= \sum_{b\in \mathcal{B}_{\underline{0}} }   [b, x] \otimes  \check{b}\cdot v      .\]
Then, by Lemma \ref{lem_eg_estimate}, for $(\fg, m)$ different from $(D_4, 3)$,  eigenvalues of $S_0$ are bounded above in absolute value by 
\begin{equation}
\label{d_mu}
 d_\mu:=  \max_{i=0,\cdots, m-1}\{    \langle \mu+2\rho_\sigma, \theta_{\underline{i} }    \rangle       \}  .   
 \end{equation}

Let $\sigma$ be a special automorphism of $\fg$ of order $m$ and let $\sigma$ act on $\mathbb{P}^1$ by $\sigma(z) = e^{\frac{2\pi i}{m}} z$ for $z\in \mathbb{P}^1$.
Let $\vec{z}=(z_0=0, z_1,\cdots, z_s)$ be distinct points in the unit disc $\mathbb{D}$ with distinct  $\Gamma$-orbits, and $\vec{\mu}=(\mu_i)_{0\leq i\leq s}$ be the weights associated to $\vec{z}$ respectively, where $\Gamma$ is the cyclic group of order $m$ generated by $\sigma$, $\mu_0$ is a dominant weight of $\fg^\sigma$ and $\mu_i \, (i\geq 1)$ is a dominant weight of $\fg$.  

Let 
\begin{equation}
\label{d_vec_mu}
d_{\vec{\mu}}:= \max_{k=1,\cdots, s} \{   d_{\mu_0},  \langle   \mu_k+2\rho,  \theta \rangle   \},
\end{equation}
where $d_{\mu_0}$ is defined in (\ref{d_mu}), and $\theta$ is the highest root of $\fg$. 

Consider the $\mathbb{C}$-linear map (for any $z\in \mathbb{D}$), 
\[\eta_z: t^{-1}\mathbb{C}[t^{-1}] \to t^{-1}\mathbb{C}[[t^{-1}]]  \]
defined by (for any $p> 0$)
\begin{equation}
\label{eta_map}
 \eta_{z}(t^{-p}):=    \frac{1}{p} \sum_{0<n<p } \left(\bar{z}^{p-n} (1-|z|^{2n})t^{-n}\right) +  \frac{1}{p}(1-|z|^{2p} )\sum_{n\geq 0} z^n t^{-n-p} . \end{equation}

 As before,  $L^-(\fg):=t^{-1}\fg[t^{-1}]$ with the Hermitian form $\{\,\}_0$ and $\bar{L}^-(\fg)$ is its $L^2$-completion. Then, it is easy to see that 
 \begin{equation} 
  T_z(x(-p) \otimes v )=\pi\left( \sum_{b\in \CB}\left([b, x]\otimes \eta_z(t^{-p})\right)\right)\otimes \check{b}\cdot v,\,\,\text{for $x(-p)\in L^-(\fg, \sigma)$ and $v\in V(\mu)_z$,}
  \end{equation} 
  where $\pi: \bar{L}^-(\fg) \to \bar{L}^-(\fg,\sigma)$ is the orthogonal projection onto the $\sigma$-invariants.  
  
  \begin{definition} {\rm For any $ z\in \mathbb{D}$, choose a M{\"o}bius transformation $\phi=\phi_z$ of the unit disc $\mathbb{D}$:
  \[ \phi(t)=\frac{t+z }{1+ \bar{z}t }   .  \]
  Then, $\phi^{-1}(t)=\frac{t-z}{1-\bar{z}t }$.  Let $H_0^{-1}$ be the operator acting on $t\mathbb{C}[[t]]$ such that 
  $H_0^{-1}(t^k)=\frac{t^k}{k}$ for any $k\geq 1$.  Then, the inverse of $H_0^{-1}$ on $t\mathbb{C}[[t]]$ is $H_0=t\frac{d}{dt}$. 
  We extend  $H_0^{-1}$ to    $\hat{H}_0^{-1}$ acting on $\mathbb{C}[[t]]$ defined by   
  \begin{equation}
  \label{newprop_2.12_conv}
  \hat{H}_0^{-1}(f(t)):=H_0^{-1} (f(t)-f(0) ), \,\,\,\text{for $f(t) \in \mathbb{C}[[t]]$}.
  \end{equation}

Following Teleman (cf. \cite[Proposition 2.5.4]{Te}), let
 $H_{z}^{-1}: L^-(\fg) \to L^+(\fg)^*$  be the operator defined by
\begin{equation} \label{eqn2.11.1} 
\left(H_{z}^{-1} (x(-p))\right)(y(n))= \langle y\otimes  \phi^{-1}\circ \hat{H}_0^{-1}\circ  \phi (t^n), x(-p)\rangle, \,\,\,\text{for $x(-p)\in L^-(\fg) $ and $y(n)\in L^+(\fg)$}.
 \end{equation}
 Then, the image of  $H_{z}^{-1}$ is contained in the completion  $\bar{L}^-(\fg)$ under the identification of   $\bar{L}^-(\fg)$
 as a subspace of  $L^+(\fg)^*$ under $\xi$, where $\xi: L^-(\fg) \to L^+(\fg)^*$ is induced from the invariant form $\langle\,,\,\rangle$ on $\hat{L}(\fg)$ (cf. Subsection \ref{cochain}).  Thus, we will think of the operator $H_z^{-1}$ as a linear map $L^-(\fg) \to \bar{L}^-(\fg)$.}
 \end{definition}

 \begin{proposition}
\label{prop_teleman} For any $z\in \mathbb{D}$, the operator  $H_{z}^{-1}$ 
 coincides with the operator 
 $x(-p)\in L^-(\fg) \mapsto 
  x\otimes \eta_{z} (t^{-p}) \in \bar{L}^{-}(\fg)$.
  \end{proposition}
  \begin{proof}
  We first  show that, for any $n\geq 1$,  we have the following equality:
  \begin{equation}\label{eqn2.12.2} ( \phi^{-1}\circ \hat{H}_0^{-1}\circ  \phi) (t^n) =\sum_{k\geq 0} a_k^n t^k , \end{equation}
  where 
  \begin{equation}\label{eqnnew2.12.3} a_k^n= \begin{cases}
   \frac{1}{k} (1-|z|^{2k} )z^{n-k}, \quad  \text{ when $1\leq k\leq  n$ }\\
   \frac{1}{k} \bar{z}^{k-n} (1-|z|^{2n}), \quad \text{when $k\geq n$ }. 
   \end{cases}
      \end{equation}
      Observe that when $k=n$, both of the above cases give the same  value.
      
 By (\ref{newprop_2.12_conv}), it suffices to show that 
\[    \phi(t)^n-z^n = \sum_{k\geq 1} a_k^n  t\frac{d}{dt} ( \phi(t)^k ) =    \sum_{k\geq 1}k a_k^n  \phi(t)^{k-1} t \phi(t)' = \sum_{k\geq 1}k a_k^n  \phi(t)^{k-1} (1-|z|^2)\frac{t}{(1+\bar{z} t)^2}  .  \]
  Equivalently, 
\begin{equation}
\label{newprop_2.12_identity}
 t^n-z^n=   \sum_{k\geq 1}k a_k^n t^{k-1} (1-|z|^2) \frac{\phi^{-1}(t)}{(1+\bar{z} \phi^{-1}(t))^2}
  = \sum_{k\geq 1}\frac{k a_k^n}{1-|z|^2} t^{k-1}(t-z)(1-\bar{z}t) .    \end{equation}
  We compare this identity on each degree of $t$.  When the degree of $t$ is 0,  we get $a_1^n=z^{n-1}(1-|z|^2)$.  
  For any $1\leq k\leq n$, we now prove the identity $a_k^n=
   \frac{1}{k} (1-|z|^{2k} )z^{n-k}$. When $k=1$, we have just confirmed it.  When $k\geq 2$, by comparing the coefficients of $t$ in (\ref{newprop_2.12_identity}), we get 
   $$a_2^n=\frac{1}{2}z^{n-2}(1-|z|^4) -\delta_{n, 1} \frac{1-|z|^2}{2z}.$$
   This proves (\ref{eqn2.12.2}) for $k=2$ and any $n\geq 1$.  
   
   Thus, we can assume $n\geq 3$. We induct on $k$. When $3\leq k\leq n$, comparing the coefficients of $t^{k-1}$, we have 
   \begin{align*}
    ka_k^n&= \frac{1}{ z} ( -\bar{z} (k-2)a^n_{k-2} + (1+|z|^2) (k-1)a^n_{k-1}   ) \\
          &=  \frac{1}{ z} ( -\bar{z}(1-|z|^{2k-4} )z^{n-k+2} +  (1+|z|^2) (1-|z|^{2k-2})z^{n-k+1}  )\\
          &=(1-|z|^{2k})z^{n-k}. 
     \end{align*}
 Now, we consider the case $k>n$. When $k=n+1$, comparing the coefficients of $t^n$ in (\ref{newprop_2.12_identity}), we get 
 \begin{align*}
  (n+1)a^n_{n+1}&=\frac{1}{z}  ( -(1-|z|^2) +  (n-1)a^n_{n-1}(-\bar{z}) +na^n_n(1+|z|^2)    )  \\
               &=  \frac{1}{z}  ( -(1-|z|^2) +   (1- |z|^{2n-2})(-|z|^2) + (1-|z|^{2n})(1+|z|^2)    )\\
               &=   \bar{z}(1-|z|^{2n}). 
  \end{align*}
 When $k=n+2$,  we have 
 \[(n+2)a^n_{n+2}=\frac{1}{z } ( na^n_n(-\bar{z})+ (n+1)a^n_{n+1} (1+|z|^2)    ) = \bar{z}^2 (1-|z|^{2n}) .    \]
 When $k\geq n+3$, inducting on $k$, we have 
 \begin{align*}  ka_k^n&= \frac{1}{z} ((k-2)a^n_{k-2} (-\bar{z})   + (k-1)a^n_{k-1}(1+|z|^2)    )  \\
            &=   \frac{1}{z}  (-\bar{z}^{k-1-n} (1-|z|^{2n}) +  \bar{z}^{k-1-n} (1-|z|^{2n} )(1+|z|^2)      )\\
            &= \bar{z}^{k-n}(1-|z|^{2n}). 
  \end{align*}
  This finishes the proof of the identity (\ref{eqn2.12.2}).

    Write 
  $$ H_{z}^{-1} (x(-p))= \sum_{n>0; b\in \mathcal{B} }\, c_{b(-n)} b(-n), \,\,\,\text{for some  $c_{b(-n)}\in \mathbb{C}$}.$$
  Pairing the right side of the above identity with $\check{b} (n)$ under   $\langle\,,\,\rangle$, we get
  $$   c_{b(-n)} =   \langle\check{b},x\rangle   \Res_{t=0} \left( \phi^{-1}\circ \hat{H}_0^{-1}\circ  \phi (t^n) t^{-p-1}\right).$$
  Thus,
\begin{align}  
  H_{z}^{-1} (x(-p))&=  \sum_{n>0; b\in \mathcal{B}}\,   \langle\check{b},x\rangle   \Res_{t=0} \left( \phi^{-1}\circ \hat{H}_0^{-1}\circ  \phi (t^n) t^{-p-1}\right) b(-n)  \notag\\
  &= \sum_{n>0}\,     \Res_{t=0} \left( \phi^{-1}\circ \hat{H}_0^{-1}\circ  \phi (t^n) t^{-p-1}\right) x(-n)  \notag\\
  &= \sum_{n>0}\, a^n_p(x(-n)), \,\,\,\text{where $a^n_p$ is given by the equation (\ref{eqnnew2.12.3})}\notag\\
  &= \frac{1}{p} \sum_{0<n<p} \,\bar{z}^{p-n} (1-|z|^{2n})x(-n) +\frac{1}{p}   (1-|z|^{2p}) \sum_{n\geq p}\,z^{n-p} x(-n).
  \end{align}
  Thus, from the above identity, the operator $ H_{z}^{-1}$ coincides with the operator 
  $$x(-p) \mapsto   x\otimes \eta_z(t^{-p}),\,\,\text{for $x(-p) \in L^-(\fg)$}.$$
  This proves the proposition. 
   \end{proof}

   We now recall the following  proposition due to  Teleman (\cite[Proposition 2.5.5]{Te}).  
\begin{proposition}
\label{prop_tel_est}
For any $A=\sum_{p\geq 1}x_p(-p)\in L^-(\fg) $ and any $\epsilon>0$, 
\[  ||  \sum_{k=0}^s   \sum_{p\geq 1}   x_p\otimes \eta_{z_k} (t^{-p}) ||_0\leq  (1+\epsilon)  ||A||_0  ,  \]
provided the points $\{z_0, \cdots, z_s\}$ are sufficiently far apart in the hyperbolic metric on the unit disc $\mathbb{D}$. 
\end{proposition}

\begin{proposition} \label{prop_norm_T_est} 
Let $\sigma$ be a special automorphism of $\fg$ of order $m$ and let $(\fg, m)$ be different from $(D_4, 3)$. 
For any given $\epsilon >0$, we can choose points $\{z_0=0, z_1, \dots, z_s\}$ in $\mathbb{D}$ far apart in the hyperbolic metric such that 
the operator $T_{\vec{z}}$ defined in (\ref{op_T_z})  satisfies   $ \mid\{ T_{\vec{z}}(A) , A  \}_0\mid\leq   q d_{\vec\mu} (1+\epsilon)  || A||_o^2 $, 
for any $A$ in $\wedge^q (L^-(\fg,\sigma)) \otimes \mathscr{H}_c(\lambda)\otimes V(\vec{\mu})_{\vec{z}}$. 
\end{proposition}
\begin{proof}
By the definition of $T_{\vec{z}}$, 
\begin{equation}\label{eqn6.12.1}  \{ T_{\vec{z}}(A), A\}_0 =\sum_{\ell=1}^q    \{ \sum_{k=0}^s  T^\ell_{z_k}  (A) , A   \}_0 , \end{equation}
where $T^\ell_{z_k}(A)$ denotes the action only on the $\ell$-th factor of $\wedge^q L^-(\fg,\sigma)$ and the factor $V(\mu_k)_{z_{k}}$. 

Let $\rho^k$ denote the representation $V(\mu_k)$ of $\fg^{\Gamma_{z_k} }$. 
Define the operator $S_k: =  \sum_{b\in \mathcal{B} } ad_b \boxtimes  \rho^k_{\check{b}}$ on $L^-(\fg) \otimes V(\mu_k)$  for $k\geq 1$ (which clearly extends to an operator on $\bar{L}^-(\fg)\otimes V(\mu_k)$); and when $k=0$, denote $S_0= \sum_{b\in \mathcal{B}_{\underline{0}}}  ad_b \boxtimes  \rho^0_{\check{b}}$ on $L^-(\fg)\otimes V(\mu_0)$.    
Consider
the operator 
\[\bar{\eta}_{z_k}: L^-(\fg)\otimes V(\mu_k)\to \bar{L}^-(\fg)\otimes V(\mu_k)\]
defined by 
\[ \bar{\eta}_{z_k}  (x(-p) \otimes v  )= (x\otimes \eta_{z_k} (t^{-p})  )\otimes v, \quad \text{ for any }  x\in \fg, p\geq 1 \text{ and } v\in V(\mu_k),  \]
where $\eta_{z_k}$ is defined in (\ref{eta_map}), and extend it multi-linearly  on $\wedge^q (L^-(\fg)) \otimes V(\mu_k)$. By Proposition \ref{prop_tel_est}, $\bar{\eta}_{z_k}$ extends uniquely to a continuous operator on 
$ \bar{L}^-(\fg)\otimes V(\mu_k)$ with respect to the Hermitian metric $\{\,,\,\}_0$ on $\bar{L}^-(\fg)$.
Then, the operator $S_k$ is  self-adjoint
and it commutes with $\bar{\eta}_{z_k}$ for any $0\leq k\leq s$. We extend both the operators $S_k^\ell$ and $  \bar{\eta}^\ell_{z_k}$ to  $ \wedge^q (\bar{L}^-(\fg)) \otimes \mathscr{H}_c(\lambda)\otimes V(\vec{\mu})_{\vec{z}}$ by defining them to be the identity operator on the $\mathscr{H}_c(\lambda)$ factor, where $S_k^\ell$ is the operator $S_k$ acting only on the $\ell$-th factor of $\wedge^q \bar{L}^-(\fg)$ and the $k$-th factor $V(\mu_k)$ of $ V(\vec{\mu})_{\vec{z}}$ and similarly for $ \bar{\eta}^\ell_{z_k}$.
Moreover, for $A, B\in \wedge^q (L^-(\fg,\sigma)) \otimes \mathscr{H}_c(\lambda)\otimes V(\vec{\mu})_{\vec{z}}$,
\begin{equation} \label{eqn6.12.2}  \{ T^\ell_{z_k}A, B   \}_0=  \{  S_k^\ell \circ  \bar{\eta}^\ell_{z_k} (A) , B   \}_0    ,
\end{equation}
since for any $x(-p), x'(-p')\in L^-(\fg,\sigma)$,  any $0\leq k\leq s$ and $v\in V(\mu_k)$,
\[   \{ \sum_{b\in \mathcal{B}}     \pi ( [b,x] \otimes  \eta_{z_k}(t^{-p})    )   \otimes \check{b} \cdot v,   x'(-p')\otimes v    \}_0  =   \{ \sum_{b\in \mathcal{B}}      ( [b,x] \otimes  \eta_{z_k}(t^{-p})    )  \otimes \check{b} \cdot v,   x'(-p')\otimes v    \}_0, \]

We next show that for any $1\leq \ell\leq  q$ and $A\in \wedge^q (L^-(\fg,\sigma)) \otimes \mathscr{H}_c(\lambda)\otimes V(\vec{\mu})_{\vec{z}}$, 
\begin{equation}
\label{key_ineq}
   \mid\{  \sum_{k=0}^s S_k^\ell \circ  \bar{\eta}^\ell_{z_k} (A) ,A   \}_0  \mid     \leq     d_{\vec{\mu}} \sum_{k=0}^s \{ \bar{\eta}_{z_k}^\ell (A), A    \}_0 .   \end{equation}
To prove this, observe first that any eigenvalue of the operator $S_k$, acting on $\fg\otimes V(\mu_k)$ for $k\geq 1$, is bounded above in absolute value by $\langle  \mu_k+2\rho, \theta  \rangle$ (by Lemma \ref{lem_eg_estimate} for $\sigma=\Id$). Further, for $z_0=0$, i.e.,  $k=0$, by Lemma \ref{lem_eg_estimate}, the operator $S_0$ acting on $\fg\otimes V(\mu_0)$ has all its eigenvalues in absolute value bounded above by $d_{\mu_0}$ defined in (\ref{d_mu}).   Thus, all the eigenvalues in absolute value of any of the operators $\{  S_k  \}_{0\leq k \leq s}$ are bounded above by $d_{\vec{\mu}}$.  

We now prove (\ref{key_ineq}).   Decompose $\wedge^q (L^-(\fg))\otimes \mathscr{H}_c(\lambda)\otimes 
V(\vec{\mu})_{\vec{z}}$ into the direct sum of eigenspaces under $S_k^\ell$ as follows: 
\[  \wedge^q(L^-(\fg))\otimes  \mathscr{H}_c(\lambda)\otimes
V(\vec{\mu})_{\vec{z}} =\perp_{\mid n\mid \leq d_{\vec{\mu} }}  E_n^k  , \]
where $E^k_n$ is the eigenspace of the self-adjoint operator $S_k^\ell$ with eigenvalue $n$. Since the operator $\bar{\eta}_{z_k}^\ell$ commutes with $S_k^\ell$ for any $0\leq k\leq s$ and any $1\leq \ell \leq q$, the eigenspaces $E^k_n$ are stable under the action of  $\bar{\eta}_{z_k}^\ell$. Write 
\[A= \sum_n A_n^k,   \,\text{ for } A_n^k\in E^k_n . \]
By orthogonality, $A_n^k\perp A_m^k$ for $n\not= m$. Thus,    we get (for any fixed $0\leq k\leq s$ and $1\leq \ell \leq q$)
\begin{align}
\label{S_k_identity}
 \mid\{ S_k^\ell \circ  \bar{\eta}^\ell_{z_k}  (A) , A \}_0 \mid   &=\mid \{ \sum_n  n \,  \bar{\eta}^\ell_{z_k} (A_n^k)   , A\}_0 \mid= \mid\sum_{\mid n\mid \leq d_{\vec{\mu}}}\, n\{   \bar{\eta}^\ell_{z_k} (A_n^k)   , A_n^k\}_0 \mid\notag\\
 & \leq   d_{{\vec{\mu}}}  \sum_n   \{ \bar{\eta}^\ell_{z_k} (A_n^k)   , A_n^k\}_0 =d_{\vec{\mu}}    \{ \bar{\eta}^\ell_{z_k}  (A) , A  \}_0 , \end{align}
since $  \bar{\eta}^\ell_{z_k}  $ is a positive operator as it is diagonalizable with all its eigenvalues $\{\frac{1}{p}\}_{p\in \mathbb{Z}_{>0} }$ (cf. \cite[Theorem 12.32]{R}).  To prove this, observe that $\bar{\eta}^\ell_{z_k} $ is conjugate of the operator $\bar{\eta}^\ell_{0} $ under the action of $\phi_{z_k}$ on the holomorphic functions on the unit disc which keeps the inner product $\{, \}_0$ invariant (use Proposition 
\ref{prop_teleman}). 
 Thus, by identity (\ref{S_k_identity}), for any $1\leq \ell\leq q$, the inequality (\ref{key_ineq}) holds.  

Thus, by Proposition \ref{prop_tel_est}, identities \eqref{eqn6.12.1} and \eqref{eqn6.12.2} (for $B=A$) and the inequality (\ref{key_ineq}), we have 
\[    \mid\{  T_{\vec{z}}A, A  \}_0  \mid = \mid\sum_{\ell=1}^q   \{   \sum_{k=0}^s  T^\ell_{z_k}  (A) , A \}_0 \mid \leq    q d_{\vec{\mu}} (1+\epsilon)   ||A||_0^2 ,\]
provided the points $\{z_k\}$ are far apart under the hyperbolic metric in the unit disc.    This proves our proposition. 

\end{proof}

We now come to the main theorem of this paper. 

\begin{theorem} \label{thm3.15} Let $\sigma$ be a special automorphism of $\fg$ of order $m$ and let $(\fg, m)$ be different from $(D_4, 3)$.
For any $q\geq 1$, the Lie algebra cohomology 
\[ H^q(L^+(\fg,\sigma),  \hat{\mathscr{H}}_c(\lambda)\otimes V(\vec{\mu})_{\vec{z}})^{\fg^\sigma}   =0, \quad \text{ if }\, d_{\vec{\mu}}< \frac{ c+2\check{h}}{m},\]
and the points $\vec{z}= \{z_0 = 0, z_1, \dots, z_s\}$ in $\mathbb{D}$ are far apart in the hyperbolic metric.
\end{theorem}
\begin{proof}
Let $A$ be a nonzero element in $[\wedge^q (L^-(\fg,\sigma)) \otimes \mathscr{H}_c(\lambda)\otimes V(\vec{\mu})_{\vec{z}}]^{\fg^\sigma} $  such that $\bar{\Box}A=0$. Then, by Theorem \ref{newnakano},   
\[0=\bar{\Box} A= \Box A+   \frac{q( c+2\check{h})  }{m} A  +T_{\vec{z}}A   .\]
This gives  
\[  - \{  T_{\vec{z}}A, A  \}_0 =    \{  \Box A, A \}_0+      \frac{q( c+2\check{h})  }{m} || A ||_0^2    ,  \]
which gives, since $\{  \Box A, A \}_0  \geq 0 $,  
\[ - \{  T_{\vec{z}}A, A  \}_0   \geq    \frac{q( c+2\check{h})  }{m}   || A ||_0^2. \]
But, by Proposition \ref{prop_norm_T_est},   
\[  \mid \{  T_{\vec{z}}A, A  \}_0 \mid  \leq     q d_{\vec{\mu}} (1+\epsilon )     || A ||_0^2.  \]
Thus, we get $   \frac{ c+2\check{h}}{m}  \leq  d_{\vec{\mu}} (1+\epsilon)   $.  By assumption, $d_{\vec{\mu}}< \frac{ c+2\check{h}}{m}$, which violates the last inequality for some small enough $\epsilon$. This proves that $A=0$. Thus, $\bar{\Box}$ has no nonzero kernel
in $ [\wedge^q (L^-(\fg,\sigma)) \otimes \mathscr{H}_c(\lambda)\otimes V(\vec{\mu})_{\vec{z}}]^{\fg^\sigma}$.  Rest of the proof of the theorem follows from \cite[Proposition 3.3.1]{Te}. 
\end{proof}

\begin{remark}
\label{remark_3.16}
{\rm (a) In fact, the above  Theorem \ref{thm3.15} remains true under no restrictions on $\vec{z}\in \bc^{s+1}$ (except that they are distinct) by following the same argument as that of the proof of \cite[Proposition 3.4.2]{Te}.

\vskip1ex
(b) Theorem \ref{thm3.15} extends \cite[Theorem 0 (c)]{Te} to the {\it twisted} situation.}
\end{remark}

Let $\sigma$ be an automorphism of $\fg$ of order $m$. Further, as in Notation \ref{notation}, let $\sigma$ act on $\mathbb{P}^1$ via 
$\sigma (z) = e^{\frac{2\pi i}{m}} z$  for $z\in \mathbb{P}^1$. Take $s+2$ points $(\infty, \vec{z})$ in $\mathbb{P}^1$,
 where $\vec{z} = (z_0=0, z_1, \dots, z_s)$, such that their $\sigma$ orbits are distinct. Let $\lambda\in D_{c, \sigma^{-1}}, \mu_0\in D_{c, \sigma}$ and $\mu_1, \dots, \mu_s\in D_c$ be attached successively to the points  $(\infty, \vec{z})$.  
Let $\Gamma$ be the cyclic group generated by $\sigma$, and let $\Gamma_z$ be the stabilizer group of $\Gamma$ at $z\in \mathbb{P}^1$.
 Let $\lambda^*$ (resp. $\mu_i^*$) denote the highest weight of the dual representation $V(\lambda)^*$ (resp. $V(\mu_i)^*$) of $\fg^\sigma$ (resp. $\fg^{\Gamma_{z_i}}$).  

\begin{lemma} \label{lemma2.17} Follow the notation as above.
\vskip1ex

(a) The space of vacua 
$$ \mathscr{V}_{\mathbb{P}^1, \Gamma}\left((\infty, \vec{z}), (\lambda, \mu_0, \mu_1, \dots, \mu_s) \right)^\dagger \simeq H^0\left(L^+(\fg,\sigma),  \hat{\mathscr{H}}_c(\lambda^*)\otimes V(\vec{\mu}^*)_{\vec{z}}\right)^{\fg^\sigma},$$
where $\vec{\mu}^*= (\mu_0^*, \mu_1^*, \dots , \mu_s^*)$.

\vskip1ex

(b) For any $\lambda\in D_{c, \sigma^{-1}}, \mu_0$ a dominant weight of $\fg^\sigma$ and  $\mu_1, \dots, \mu_s$
any dominant weights of $\fg$ (we do not require $\mu_0\in D_{c, \sigma}$ nor do we require  $\mu_1, \dots, \mu_s\in D_c$) and any $q\geq 0$,
$$ H^q\left(L^+(\fg,\sigma),  \hat{\mathscr{H}}_c(\lambda^*)\otimes V(\vec{\mu}^*)_{\vec{z}}\right)\simeq \left(H_q(L^+(\fg,\sigma),  \mathscr{H}_c(\lambda)_{\infty} \otimes V(\vec{\mu})_{\vec{z}})\right)^*$$
as $\fg^\sigma$-modules, where $\mathscr{H}_c(\lambda)_{\infty}$ is the highest weight integrable module of $\hat{L}(\fg, \sigma^{-1})$ attached to $\infty\in \mathbb{P}^1$.
\end{lemma}
\begin{proof} 

(a) 
To prove the (a)-part, by the Propagation of Vacua (cf. \cite[Theorem 4.3]{HK1}), 
\begin{align*} 
 \mathscr{V}_{\mathbb{P}^1, \sigma}\left((\infty, \vec{z}), (\lambda, \mu_0, \mu_1, \dots, \mu_s) \right)^\dagger &= \left[\mathscr{H}_c(\lambda)^*_\infty\otimes V(\mu_0)^*_{{z}_0}\otimes \dots \otimes  V(\mu_s)^*_{{z}_s}\right]^{\fg\otimes \mathbb{C}[\mathbb{P}^1\setminus \infty]}\\
 &\simeq   \left[\hat{\mathscr{H}}_c(\lambda^{*})\otimes V(\mu_0^{*})_{z_0}\otimes V(\mu_1^*)_{z_1} \otimes \dots \otimes 
 V(\mu_s^*)_{z_s} 
 \right]^{\fg[t]^\sigma}, \,\,\text{by using (\ref{neweqn2.5.1})}\\
 & =H^0\left(L^+(\fg,\sigma),  \hat{\mathscr{H}}_c(\lambda^{*})\otimes V(\vec{\mu}^*)_{\vec{z}}\right)^{\fg^\sigma}.
 \end{align*} 
 This proves (a)
 
 \vskip1ex
 
 (b) By \cite[Lemmas 3.1.9 and 3.1.13]{Kbook} and the identity (\ref{neweqn2.5.1}), 
 \begin{equation} \label{eqn2.17.1}
 H^q\left(L^+(\fg,\sigma),  \hat{\mathscr{H}}_c(\lambda^{*})\otimes V(\vec{\mu}^*)_{\vec{z}}\right)\simeq  \left(H_q\left(L^+(\fg,\sigma),  \mathscr{H}^{\low}_c(\lambda)\otimes V(\vec{\mu})_{\vec{z}}\right)\right)^*,\,\,\text{as $\fg^\sigma$-modules}.
 \end{equation}
 Observe that $ \mathscr{H}^{\low}_c(\lambda)$ is exactly $\mathscr{H}_c(\lambda)_\infty$ with respect to the action of $\hat{L}(\fg,\sigma^{-1})$, where $L^+(\fg, \sigma)=(t\fg[t])^\sigma$ is naturally the negative part of $\hat{L}(\fg,\sigma^{-1})$ with $t_\infty=t^{-1}$ as the coordinate at $\infty$.  This conclude the proof of part (b) of the lemma. 
 
  \end{proof}

  
   
  We now assume $\sigma$ is a special automorphism on $\fg$.  
  \begin{corollary} \label{coro3.18}  Under the assumptions of Theorem \ref{thm3.15}, for any $q\geq 1$ we have
  \[ H_q(L^-(\fg,\sigma^{-1}),  \mathscr{H}_c(\lambda^*)\otimes V(\vec{\mu}^*)_{\vec{z}^{-1} } )^{\fg^\sigma}   =0, \quad \text{ if }\, d_{\vec{\mu}}< \frac{ c+2\check{h}}{m},\]
where $\vec{z}^{-1} =(\infty, z_1^{-1}, \cdots, z_s^{-1})$, 
$\mathscr{H}_c(\lambda^*)$ is the highest weight module of $\hat{L}(\fg,\sigma^{-1})$,
and $V(\vec{\mu}^*)_{\vec{z}^{-1} }:=V(\mu_0^*)_{\infty}\otimes \cdots  \otimes V(\mu_s^*)_{z_s^{-1}} $ represents the evaluation representation of $L^-(\fg,\sigma^{-1})$ in the usual sense.  
  \end{corollary}    
  \begin{proof}
  
  Observe that there is an isomorphism of Lie algebras $\gamma: L^-\fg,\sigma)\to L^+(\fg, \sigma^{-1})$ given by $x\otimes f(t)\mapsto x\otimes f(t^{-1})$. Note that in this isomorphism,  by convention on LHS $\sigma(t)=e^{\frac{-2\pi i}{m}}t$ and on RHS $\sigma(t)= e^{\frac{2\pi i}{m}}t $.  With respect to the isomorphism $\gamma$, in part (b) of Lemma \ref{lemma2.17},  the action of $L^+(\fg,\sigma)$ on $\mathscr{H}_c(\lambda^*)_\infty$ corresponds to the action of $L^-(\fg,\sigma^{-1})$ on $\mathscr{H}_c(\lambda^*)$ as a highest weight module of $\hat{L}(\fg, \sigma^{-1})$;  the evaluation representation $V(\mu_i)_{z_i}$ of $L^+(\fg, \sigma)$ corresponds to the evaluation representation $V(\mu_i)_{z_i^{-1}}$ of $L^-(\fg, \sigma)$.  Then, from part (b) of Lemma \ref{lemma2.17}, we get 
  $$ H^q\left(L^+(\fg,\sigma),  \hat{\mathscr{H}}_c(\lambda^*)\otimes V(\vec{\mu}^*)_{\vec{z}}\right)\simeq \left(H_q(L^-(\fg,\sigma^{-1}),  \mathscr{H}_c(\lambda) \otimes V(\vec{\mu})_{\vec{z}^{-1}})\right)^*. $$
Finally, our corollary follows from Theorem \ref{thm3.15} and Remark \ref{remark_3.16} (a).

  \end{proof}
  
  \begin{remark}
  Corollary \ref{coro3.18} verifies the conjecture in \cite[Conjecture 5.6]{HK2} under some strong constraints. 
  \end{remark}

  \section{Laplacian Computation}\label{section3}

{\it Throughout this section $\sigma$ is any automorphism of $\fg$ of order $m$ except in the last part of subsection \ref{nakanoproof}, where we restrict $\sigma$ to be a special automorphism. We abbreviate $\hat{\mathcal{B}}_{\sigma}$ (resp. $\hat{\mathcal{B}}_{\sigma}^q$) by $\hat{\mathcal{B}}$ (resp. 
$\hat{\mathcal{B}}^q$).}

\vskip1ex

We follow the notation from $\S$\ref{cochain} and first calculate the differential $d^q: C^q\to C^{q+1}$ restricted to $\wedge ^q \left(L^-(\fg, \sigma)\right) \otimes \mathscr{H}_c(\lambda)\otimes V(\vec{\mu})_{\vec{z}}$ with image inside 
${\wedge}^{q+1} \left(\bar{L}^-(\fg, \sigma)\right) \otimes \mathscr{H}_c(\lambda)\otimes V(\vec{\mu})_{\vec{z}}$.

\begin{lemma}\label{lem7.1}
For  $\vec{x}(-\vec{p})\in \wedge^q\left(L^-(\fg, \sigma)\right) $ and $v\in \mathscr{H}_c(\lambda)\otimes V(\vec{\mu})_{\vec{z}}$,
$$
d^q \left( \vec{x}(-\vec{p})\otimes v\right) =\frac{1}{2}\sum_{b(-k)\in\hat{\mathcal{B}}} b(-k)\wedge \left( \overline{ad}_{\check{b}(k)} \vec{x}(-\vec{p})\right)\otimes v
+\sum_{b(-k)\in\hat{\mathcal{B}}} b(-k)\wedge \vec{x}(-\vec{p})\otimes \check{b}(k)\cdot v,$$
where, as before,  $\overline{ad}$ is the projection of $ad$ to $L^-(\fg,\sigma)$.
\end{lemma}

\begin{proof} Let $d^q_1$ be the same differential $d^q$ but with the trivial action of $L^+(\fg, \sigma)$ on 
 $\mathscr{H}_c(\lambda)\otimes V(\vec{\mu})_{\vec{z}}$. Then, 
\begin{align*}
d_1^q &\left(\vec{x}(-\vec{p})\otimes v \right) \left(\check{b}_1(k_1)\wedge \dots \wedge \check{b}_{q+1}(k_{q+1})\right)\\
&=\left(\vec{x}(-\vec{p})\otimes v \right) \left( \frac{1}{2}\sum_{i=1}^{q+1} (-1)^i \ad_{\check{b}_i(k_i)}\cdot \left(
\check{b}_1(k_1)\wedge \dots \wedge\widehat{\check{b}_i(k_i)}\wedge \dots \wedge \check{b}_{q+1}(k_{q+1})\right)\right).
\end{align*}
Thus,
\begin{align}
d_1^q &
\left(\vec{x}(-\vec{p})\otimes v \right)=\frac{1}{2}\sum_{i; \vec{b}(\vec{k})\in\hat{\mathcal{B}}^{q+1}} (-1)^{i+1}\left\langle \vec{x} (-\vec{p}),ad_{\check{b}_i(k_i)}\left(
\check{b}_1(k_1)\wedge \dots \wedge\widehat{\check{b}_i(k_i)}\wedge \dots \wedge \check{b}_{q+1}(k_{q+1})\right)\right\rangle\notag\\
&\quad \quad\quad\quad\quad b_1 (-k_1)\wedge \cdots \wedge b_{q+1}(-k_{q+1})\otimes v\notag\\
&=\frac{1}{(q+1!)2} \sum_{i; b_j(-k_j)\in\hat{\mathcal{B}}}\,
\left \langle \left( ad_{\check{b}_i(k_i)}\vec{x}(-\vec{p})\right) , \check{b}_1(k_1) \wedge \cdots \wedge \widehat{\check{b}_i(k_i)}\wedge \cdots \wedge \check{b}_{q+1} (k_{q+1})\right\rangle\notag \\
&\qquad b_{i}(-k_i)\wedge b_{1}(-k_1)\wedge \cdots \wedge \widehat{b_{i}(-k_i)}\wedge \cdots \wedge b_{q+1}(-k_{q+1}) \otimes v,\,\,
\text{ by Lemma \ref{lem_herm2}}\notag\\
&=\frac{1}{2}\sum_{b(-k)\in\hat{\mathcal{B}}}b(-k)\wedge \left( \overline{ad}_{\check{b}(k)}\vec{x}(-\vec{p})\right)\otimes v\label{eqn7.11}
\end{align}
Further,
\begin{align*}
\left(d^q-d_1^q\right) \left(\vec{x}(-\vec{p})\otimes v\right)
& \left(\check{b}_1(k_1)\wedge \dots \wedge \check{b}_{q+1}(k_{q+1})\right)
\\
&=\sum_{i=1}^{q+1}\,(-1)^{i+1}\left\langle \vec{x}(-\vec{p}), 
\check{b}_1(k_1)\wedge \dots \wedge\widehat{\check{b}_i(-k_i)}\wedge \dots \wedge \check{b}_{q+1}(k_{q+1})
\right\rangle \check{b}_{i}(k_i)\cdot v
\end{align*}
Thus,
\begin{align}
\left(d^q-d_1^q\right)&\left(\vec{x}(-\vec{p})\otimes v \right)
 =\notag\\
&\sum_{i; \vec{b}(\vec{k})\in\hat{\mathcal{B}}^{q+1}} (-1)^{i} \left \langle \vec{x} (-\vec{p}),\check{b}_1(k_1)\wedge \cdots \wedge \widehat{\check{b}_i(k_i)}\right.
\left.\wedge \cdots \check{b}_{q+1} (k_{q+1})\right\rangle\notag\\
&\qquad \cdot b_1(-k_1)\wedge \cdots \wedge b_i(-k_i)\wedge \cdots \wedge 
b_{q+1}(-k_{q+1})\otimes \left(\check{b}_i(k_i) \cdot v\right)\notag\\
&=\sum_{i; b_i(-k_i)\in\hat{\mathcal{B}}}\,\frac{q!}{(q+1)!}b_i(-k_i)\wedge \vec{x}(-\vec{p})\otimes\check{b}_i(k_i)\cdot v\notag\\
&=\sum_{b(-k)\in\hat{\mathcal{B}}} b(-k)\wedge \vec{x}(-\vec{p})\otimes \left( \check{b}(k)\cdot v\right)\label{eqn7.12}
\end{align}
Combining the equations \eqref{eqn7.11} and \eqref{eqn7.12}, we get the lemma. 
\end{proof}

\begin{lemma}\label{lem6.2.3}
For any $Y\in \wedge ^{q+1}\left(L^-(\fg, \sigma)\right)$, $Z\in \wedge ^{q}\left(L^-(\fg, \sigma)\right)$ and $x\in L^-(\fg, \sigma)$,
$$
\left\{ i_x Y,Z\right\} = \left\{ Y, x\wedge Z\right\},
$$
where $i_xY$ is the graded derivation such taht 
$$
i_x(x')=\left\{ x',x\right\}, \mbox{ for } x'\in L^-(\fg, \sigma).
$$
\end{lemma}
\begin{proof}
We can assume that $Y=y_1\wedge \cdots \wedge y_{q+1},y_i\in L^-(\fg, \sigma)$ and $Z=z_1\wedge \cdots \wedge z_q$. Then,
\begin{equation}\label{eqn7.1}
\left\{ i_x Y,Z\right\} =\sum_{p=1}^{q+1}(-1)^{p-1}\left\{y_p,x\right\} \det \left( \left\{ \overline{y_i}, z_j\right\} _{1\leq i,j\leq q}\right),
\end{equation}
where 
\begin{align*} \overline{y_i}&=y_i\,\,\text{ for $i<p$, and }\\
&= y_{i+1} \,\,\mbox{for } i\geq p.
\end{align*}

Similarly,
\begin{equation}\label{eqn7.2}
\left\{Y,x\wedge Z\right\}=\det 
\left(
\begin{array}{llll}
\left\{y_1,x \right\},&\left\{y_1,z_1 \right\},& \ldots ,&\left\{y_1,z_{q} \right\}\\
\left\{y_2,x \right\},&\left\{y_2,z_1 \right\},& \ldots ,&\left\{y_2,z_{q} \right\}\\
\vdots&&& \vdots \\
\left\{y_{q+1},x \right\},&\left\{y_{q+1},z_1 \right\}, &\ldots ,&\left\{y_{q+1},z_{q} \right\}
\end{array}
\right).
\end{equation}
Combining the equations \eqref{eqn7.1}  and \eqref{eqn7.2}. we get the lemma.
\end{proof}

\begin{lemma}\label{lem6.2.4}
For $x\in L^-(\fg, \sigma)$ and $y\in \hat{L}(\fg,\sigma)$, 
$$i_x({\overline{ad}}^\ell_y)=({\overline{ad}}^\ell_y)i_x+i^\ell_{[x,\kappa (y)]_-},\,\,\,\text{ for $1\leq \ell\leq q$,
 acting on $\wedge^q(L^-(\fg,\sigma))$}, $$
 where $i^\ell$  is the contraction action only on the $\ell$-th factor with sign $(-1)^{\ell-1}$ and $[\,\,]_-$ denotes the projection onto $L^-(\fg, \sigma)$.
\end{lemma}
\begin{proof}
Take $Z=z_1\wedge \cdots \wedge z_q, z_i \in L^-(\fg, \sigma)$.
Then,
\begin{align*}
i_x ({\overline{ad}}^\ell_y)(Z)&= \sum_{j\neq \ell}(-1)^{j-1} z_1 \wedge \dots \wedge  \left\{ z_j,x\right\}\wedge \cdots \wedge \left[ y,z_\ell\right]_-\wedge \dots \wedge z_q \\
&+(-1)^{\ell-1}z_1\wedge \dots \wedge  \left\{ \left[ y,z_\ell\right],x\right\}\wedge \cdots \wedge z_q.
\end{align*}
Further,
\begin{equation}
\left({\overline{ad}}^\ell_y\right) i_x (Z) = \sum_{j\neq \ell}(-1)^{j-1} z_1 \wedge  \cdots \wedge \left\{ z_j,x\right\}\wedge \cdots \wedge [y,z_\ell]_-\wedge  \cdots \wedge z_q.
\end{equation}
By Lemma \ref{lem_herm2}, the lemma follows.
\end{proof}

We consider the Hermitian form $\{\,,\,\}_0$ on $ \wedge ^{*}\left(L^-(\fg,\sigma)\right)$ and the standard Hermitian forms 
on $\mathscr{H}_c(\lambda)$ and $ V_{\vec{z}}(\vec{\mu})$  as in $\S$\ref{cochain}.

 \begin{lemma} \label{lem7.4}For $\vec{x}(-\vec{p})= x_1(-p_1)\wedge \dots \wedge x_{q+1}(-p_{q+1}) \in \wedge ^{q+1}\left(L^-(\fg,\sigma)\right)$ and $ v \in \mathscr{H}_c(\lambda)\otimes V_{\vec{z}}(\vec{\mu}),$
\begin{align*}
{d}^*\left(\vec{x}(-\vec{p})\otimes v\right)=\frac{1}{2}\sum_{j; b(-k)\in\hat{\mathcal{B}}}\left(\frac{1}{p_j}+\frac{1}{k}\right) i_{b(-k)}\left(ad^j_{b(-k)} \vec{x}(-\vec{p})\right)\otimes v
+\sum_{b(-k)\in\hat{\mathcal{B}}}\frac{1}{k}\left(i_{b(-k)}\vec{x}(-\vec{p})\right)\otimes b (-k)\cdot v.
\end{align*}
\end{lemma}

\begin{proof} As in $\S$\ref{chain}, $\{v_\phi\}$ is an orthonormal basis of $\mathscr{H}_c(\lambda)\otimes   V_{\vec{z}}(\vec{\mu})$. 
\begin{align*}
d^*\left( \vec{x}(-\vec{p})\otimes v\right)&=\sum_{v_\phi; \vec{a}(-\vec{n})\in\hat{\mathcal{B}}^q} \,\vert \vec{n}\vert \left\{ d^* \left(\vec{x}(-\vec{p})\otimes v\right), \vec{a}(-\vec{n})\otimes v_{\phi}\right\}
\vec{a}(-\vec{n})\otimes v_{\phi}\\
&=\sum \frac{\vert \vec{n}\vert}{\vert \vec{p}\vert} \left\{  \vec{x}(-\vec{p})\otimes v,  d\left(\vec{a}(-\vec{n})\otimes v_{\phi}\right)\right\}\vec{a}(-\vec{n})\otimes v_{\phi}\\
&=\frac{1}{2}\sum_{v_\phi; \vec{a}(-\vec{n})\in\hat{\mathcal{B}}^q; b(-k)\in\hat{\mathcal{B}}}\, \frac{\vert \vec{n}\vert}{\vert \vec{p}\vert} \left\{  \vec{x}(-\vec{p})\otimes v,b(-k)\wedge \left( \overline{ad}_{\check{b}(k)}\vec{a}(-\vec{n})\right)\otimes v_{\phi}\right\} \vec{a}(-\vec{n})\otimes v_{\phi}\\
&+\sum_{v_\phi; \vec{a}(-\vec{n})\in\hat{\mathcal{B}}^q; b(-k)\in\hat{\mathcal{B}}}\,
\frac{\vert \vec{n}\vert}{\vert \vec{p}\vert} \left\{  \vec{x}(-\vec{p})\otimes v,b(-k)\wedge\vec{a}(-\vec{n})\otimes \check{b}(k) \cdot v_{\phi}\right\}\vec{a}(-\vec{n})\otimes v_{\phi},\,\,\text{by Lemma \ref{lem7.1}}\\
&=\frac{1}{2}\sum \frac{\vert \vec{n}\vert}{\vert \vec{p}\vert} \left\{\left(i_{b(-k)} \vec{x} (-\vec{p})\right)\otimes v, \overline{ad}_{\check{b}(k)} \vec{a}(-\vec{n})\otimes v_{\phi}\right\}\vec{a}(-\vec{n})\otimes v_{\phi}\\
&+\sum \frac{\vert \vec{n}\vert}{\vert \vec{p}\vert} \left(\left\{i_{b(-k)} \vec{x} (-\vec{p})\right)\otimes v,  \vec{a}(-\vec{n})\otimes \check{b}(k)\cdot v_{\phi}\right\}\vec{a}(-\vec{n})\otimes v_{\phi}, \,\,\text{by Lemma \ref{lem6.2.3}}\\
&=\frac{1}{2}\sum \frac{\vert \vec{n}\vert}{\vert \vec{p}\vert} \left\{ ad_{b(-k)}i_{b(-k)}\vec{x}(-\vec{p})\otimes v, \vec{a}(-\vec{n})\otimes v_{\phi}\right\} \vec{a}(-\vec{n})\otimes v_{\phi}\\
&+\sum \frac{\vert \vec{n}\vert}{\vert \vec{p}\vert} \left\{i_{b(-k)} \vec{x} (-\vec{p}),  \vec{a}(-\vec{n})\right\} \vec{a}(-\vec{n})\otimes  \left\{ v , \check{b}(k)\cdot v_{\phi}\right\} v_{\phi},\,\,\text{by Lemma \ref{lem_herm2}}\\
&= \frac{1}{2}\sum_{j; b(-k)\in\hat{\mathcal{B}}}\left(\frac{1}{p_j}+\frac{1}{k}\right) i_{b(-k)}\left(ad^j_{b(-k)} \vec{x}(-\vec{p})\right)\otimes v\\
&+\sum_{b(-k)\in\hat{\mathcal{B}}}\frac{1}{k}\left(i_{b(-k)}\vec{x}(-\vec{p})\right)\otimes b (-k)\cdot v,
 \,\,\text{by Lemmas \ref{lem6.2.4} and \ref{lem_hem1}}.
 \end{align*}
 This proves the lemma. 
\end{proof}
Combining the Lemmas \ref{lem7.1} and \ref{lem7.4}, we get the following.

\begin{proposition}\label{prop7.5}
$$\overline{\Box} :=dd^*+d^*d: \wedge ^q\left(L^-\left( \fg,\sigma\right)\right)\otimes \mathscr{H}_c(\lambda) \otimes V_{\vec{z}}(\vec{\mu})\rightarrow  {\wedge}^q \left(\bar{L}^-\left( \fg,\sigma\right)\right)
\otimes \mathscr{H}_c(\lambda)\otimes V_{\vec{z}}(\vec{\mu})$$
 is given as follows.

For  $\vec{x}(-\vec{p})\in\wedge ^q\left(L^-\left( \fg,\sigma\right)\right)$ and $v = v_1\otimes v_2\in \mathscr{H}_c(\lambda) \otimes V_{\vec{z}}(\vec{\mu})$,
\begin{align}
\overline{\Box}&\left(\vec{x}(-\vec{p})\otimes v\right)= \notag\\
 &\frac{1}{4}\sum_{1\leq \ell\leq q; 1\leq j\leq q+1; b(-k), a(-n)\in\hat{\mathcal{B}}}\frac{p_j(\ell)+n}{np_j(\ell)}i_{a(-n)} ad^j_{a(-n)}
\left(b(-k)\wedge (\overline{ad}^\ell_{\check{b}(k)} \vec{x}(-\vec{p}))\right)\otimes v \label{eqn7.5.1}\\
&+\frac{1}{2} \sum_{b(-k), a(-n)\in\hat{\mathcal{B}}}
\frac{1}{n}i_{a(-n)}\left(b(-k)\wedge \left( \overline{ad}_{\check{b}(k)} \vec{x}(-\vec{p})\right)\right)\otimes a(-n)\cdot v \label{eqn7.5.2}\\
&-\frac{1}{2}\sum_{b(-k), a(-n)\in\hat{\mathcal{B}}}
\frac{k+n}{kn}\left[ a (-n), b (-k)\right] \wedge \left( i_{a(-n)}\cdot \vec{x}(-\vec{p})\right)\otimes \check{b}(k)\cdot v \label{eqn7.5.3}\\
 &+\frac{1}{2}\sum_{j>1; b(-k), a(-n)\in\hat{\mathcal{B}}}\frac{p_{j-1}+n}{np_{j-1}}i_{a(-n)} ad^j_{a(-n)} \left( b (-k) \wedge \vec{x}(-\vec{p})\right)\otimes \check{b}(k)\cdot v \label{eqn7.5.4}\\
&+\sum_{b(-k), a(-n)\in\hat{\mathcal{B}}}
\frac{1}{n}i_{a(-n)} \left( b (-k) \wedge \vec{x}(-\vec{p})\right)\otimes a (-n)\check{b}(k)\cdot v \label{eqn7.5.5}\\
&+\frac{1}{4}\sum_{j; b(-k), a(-n)\in\hat{\mathcal{B}}}
 \frac{p_j+n}{np_j} b(-k)\wedge \overline{ad}_{\check{b}(k)}\left( i_{{a}(-n)}ad^j_{a(-n)} \vec {x}(-\vec{p})\right)\otimes v \label{eqn7.5.6}\\
&+\frac{1}{2}\sum_{j; b(-k), a(-n)\in\hat{\mathcal{B}}}
 \frac{p_j+n}{np_j} b(-k)\wedge  i_{{a}(-n)}\left(ad^j_{a(-n)} \vec {x}(-\vec{p})\right)\otimes \check{b}(k)\cdot v \label{eqn7.5.7}\\
&+\frac{1}{2} \sum_{b(-k), a(-n)\in\hat{\mathcal{B}}}
\frac{1}{n}b(-k)\wedge  \left(\overline{ad}_{\check{b}(k)} i_{{a}(-n)}  \vec {x}(-\vec{p})\right)\otimes 
a (-n)\cdot v \label{eqn7.5.8}\\
&+\sum_{b(-k), a(-n)\in\hat{\mathcal{B}}}
\frac{1}{n} b(-k)\wedge  \left(i_{a(-n)}\vec{x}(-\vec{p})\right)\otimes \check{b}(k) a(-n)\cdot v,\label{eqn7.5.9}
\end{align}
where $p_1(\ell)=k$;  $p_j (\ell)=p_{j-1}$ for $j>1$ and $j\neq \ell+1$; 
$p_{\ell+1}(\ell)=p_\ell-k$ if $p_\ell>k$ otherwise $0$.
\end{proposition}

\begin{proof}
\begin{align*}
\overline{\Box}\left(\vec{x}(-\vec{p})\otimes v\right)
&=\frac{1}{2}d^* \sum_{b(-k)\in\hat{\mathcal{B}}} b(-k) \wedge \left( \overline{ad}_{\check{b}(k)} \vec{x}(-\vec{p})\right)\otimes v +d^* \sum_{b(-k)\in\hat{\mathcal{B}}}
 b(-k)\wedge \vec{x} (-\vec{p})\otimes \check{b}(k)\cdot v\\
&+\frac{1}{2} d\sum_{j\neq \ell; b(-p_\ell)\in\hat{\mathcal{B}}}\, \left(\frac{1}{p_j}+\frac{1}{p_\ell}\right)i_{b(-p_\ell)}^\ell (ad ^j_{b(-p_\ell)} \vec{x}(-\vec{p}))\otimes v\\
&+d\sum_{b(-k)\in\hat{\mathcal{B}}}\frac{1}{k}i_{{b}(-k)}\vec{x}(-\vec{p})\otimes b(-k)\cdot v, \mbox{ by Lemmas \ref{lem7.1} and \ref{lem7.4}}.
\end{align*}
Again using the Lemmas \ref{lem7.1} and \ref{lem7.4}, we get the proposition.
\end{proof}
We now simplify the terms \eqref{eqn7.5.1} - \eqref{eqn7.5.9} of Proposition \ref{prop7.5}. 

\begin{lemma} \label{lem7.6} With the notation as in  Proposition \ref{prop7.5}, we have the following, where $m$ is the order of the automorphism $\sigma$ and $\mathcal{B}_{\underline{0}}$ is an orthonormal basis of $\fg_{\underline{0}}$ as in Section \ref{section2.1}.
$$ \eqref{eqn7.5.2} +\eqref{eqn7.5.8}
=\sum_{b(-k)\in\hat{\mathcal{B}}}
\frac{1}{k}\overline{ad}_{\check{b}(k)}\vec{x}(-\vec{p})\otimes b(-k)\cdot v.$$
\begin{align*}
&\eqref{eqn7.5.5} +\eqref{eqn7.5.9}\\
&=-\sum_{ \ell; a(-n)\in \hB}\frac{1}{p_\ell}\left(ad^\ell_{a(-n)}\vec{x}(-\vec{p})\right)\otimes \left( \check{a}(n)\cdot v_1\right) \otimes v_2 - \sum_{\ell; a(-n)\in \hB}\frac{1}{p_\ell}\left(\overline{ad}^\ell_{\check{a}(n)}\vec{x}(-\vec{p})\right)\otimes \left( a(-n)\cdot v_1\right) \otimes v_2\\
&+\vec{x}(-\vec{p})\otimes \frac{q}{m} c\cdot v_1 \otimes v_2
-\sum_{\ell; b_0\in \mathcal{B}_{\underline{0}}}\frac{1}{p_\ell}\left(ad^\ell_{b_o} \vec{x}(-\vec{p})\right)\otimes (\check{b}_0\cdot v_1)\otimes v_2\\
&-\sum_{\ell; a(-n)\in \hB}\frac{1}{p_\ell}\left(ad^\ell_{a(-n)}\vec{x}(-\vec{p})\right)\otimes v_1 \otimes (z\bar{z})^{p_\ell} z^n \check{a}\cdot v_2 - \sum_{\ell; a(-n)\in \hB}\frac{1}{p_\ell}\left(\overline{ad}^\ell_{\check{a}(n)}\vec{x}(-\vec{p})\right)\otimes v_1 \otimes (z\bar{z})^{p_\ell-n}\bar{z}^n a\cdot v_2\\
&-\sum_{\ell; b_0\in \mathcal{B}_{\underline{0}}}\frac{1}{p_\ell}\left(ad^\ell_{b_o} \vec{x}(-\vec{p})\right)\otimes v_1\otimes   (z\bar{z})^{p_\ell}\check{b}_0\cdot v_2
+\sum_{a(-n)\in \hB} \frac{1}{n} \,\vec{x}(-\vec{p})\otimes a(-n) \check{a}(n)\cdot v.
\end{align*}
\begin{align*}\eqref{eqn7.5.3} +\eqref{eqn7.5.4} +\eqref{eqn7.5.7}&= \sum_{\ell; b(-k)\in\hat{\mathcal{B}}}\frac{p_\ell+k}{p_\ell\cdot k} ad^\ell_{b(-k)} \vec{x}(-\vec{p})\otimes \check{b}(k)\cdot v.\end{align*}
\begin{align*}
&\eqref{eqn7.5.1} +\eqref{eqn7.5.6}\\
&=\sum_{j\neq \ell; b(-k)\in\hat{\mathcal{B}}}\left( \frac{1}{k}-\frac{1}{p_\ell}\right) ad^j_{b(-k)}\overline{ad}^\ell_{\check{b}(k)}\vec{x}(-\vec{p})\otimes v
+\frac{1}{2}\sum_{\ell; b(-k)\in\hat{\mathcal{B}}} \frac{p_\ell}{\left(p_\ell-k\right)\cdot k}ad^\ell_{b(-k)}\overline{ad}^\ell_{\check{b}(k)} \vec{x}(-\vec{p})\otimes v\\
&-\frac{1}{2}\sum_{j\neq \ell; b_0\in \Bo} \left(\frac{1}{p_j}+\frac{1}{p_\ell}\right)ad^j_{b_0} ad^\ell_{\check{b}_0} \vec{x}(-\vec{p})\otimes v.
\end{align*}
\end{lemma}
\begin{proof}
\begin{align*}
 \eqref{eqn7.5.2} &+\eqref{eqn7.5.8}\\
&=\frac{1}{2}\sum_{b(-k)\in\hat{\mathcal{B}}}
\frac{1}{k}\overline{ad}_{\check{b}(k)} \vec{x}(-\vec{p})\otimes b(-k)\cdot v
+\frac{1}{2}\sum _{a(-n)\in\hat{\mathcal{B}}}
 \frac{1}{n}\overline{ad}_{\check{a}(n)}\vec{x}(-\vec{p})\otimes a(-n)\cdot v\\
&=\sum_{b(-k)\in\hat{\mathcal{B}}}
\frac{1}{k}\overline{ad}_{\check{b}(k)}\vec{x}(-\vec{p})\otimes b(-k)\cdot v.
\end{align*}
\begin{align*}
&\eqref{eqn7.5.5} +\eqref{eqn7.5.9}\\
&=\sum_{a(-n), b(-k)\in \hB}\frac{1}{n}i_{a(-n)}\left(b(-k)\wedge \vec{x}(-\vec{p})\right) \otimes a(-n)\check{b}(k)\cdot v
+\sum_{a(-n), b(-k)\in \hB}
\frac{1}{n}b(-k)\wedge \left(i_{a(-n)} \vec{x}(-\vec{p})\right) \otimes \check{b}(k) a(-n)\cdot v\\
&=\sum_{a(-n)\in \hB}\frac{1}{n} \vec{x}(-\vec{p})\otimes a(-n)\check{a}(n)\cdot v
-\sum_{a(-n), b(-k)\in \hB}
\frac{1}{n}b(-k)\wedge  \left(i_{a(-n)}\vec{x}(-\vec{p})\right) \otimes a(-n)\check{b}(k)\cdot v\\
&+\sum_{a(-n), b(-k)\in \hB}
\frac{1}{n}b(-k)\wedge  \left(i_{a(-n)}\vec{x}(-\vec{p})\right) \otimes \check{b}(k)a(-n)\cdot v\\
&=\sum_{a(-n)\in \hB}\frac{1}{n} \vec{x}(-\vec{p})\otimes a(-n)\check{a}(n)\cdot v
-\sum_{a(-n), b(-k)\in \hB}
\frac{1}{n}b(-k)\wedge  \left(i_{a(-n)}\vec{x}(-\vec{p})\right) 
\otimes \left(a(-n)\check{b}(k)\cdot v-\check{b}(k)a(-n)\cdot v\right)\\
&=\sum_{a(-n)\in \hB}\frac{1}{n} \vec{x}(-\vec{p})\otimes a(-n)\check{a}(n)\cdot v\\
&- \sum_{n>0; \ell; b(-k), a(-n)\in \hB}\frac{1}{n}x_1(-p_1)\wedge \quad\wedge b(-k)^{[\ell]}\wedge \cdots \wedge x_q(-p_q)\otimes \{x_\ell(-p_\ell), a(-n)\}\left(a(-n)\check{b}(k)\cdot v-\check{b}(k)a(-n)\cdot v\right)\\
&\qquad \text{ where $(\,)^{[\ell]}$ denotes the enclosed item placed in $\ell$-th place} \\
&= \sum_{a(-n)\in \hB}\frac{1}{n} \vec{x}(-\vec{p})\otimes a(-n)\check{a}(n)\cdot v\\
&- \sum_{\ell; b(-k)\in \hB}\left( \frac{1}{p_\ell}x_1(-p_1)\wedge \cdots \wedge b(-k)^{[\ell]}\wedge \cdots \wedge x_q (-p_q)
\otimes \left(\left[x_\ell(-p_\ell),\check{b}(k)\right]\cdot v_1\right)\otimes v_2
+ v_1 \otimes \left[\theta (x_\ell(-p_\ell)),\theta (\check{b}(k))\right]\cdot v_2 \right ),\\
&\qquad\qquad \mbox{where $\theta$ denotes the action on $V_{\vec{z}}(\vec{\mu})$}\\
&= \sum_{a(-n)\in \hB}\frac{1}{n} \vec{x}(-\vec{p})\otimes a(-n)\check{a}(n)\cdot v
-\sum_{\ell; b(-k)\in \hB}\frac{1}{p_\ell}x_1(-p_1)\wedge \cdot \wedge b(-k)^{[\ell]}\wedge \cdot \wedge x_q(-p_q)\\
&\otimes \bigg(\sum_{a(-n)\in \hB}\left\{[x_\ell(-p_\ell), \check{b}(k)], \check{a}(n)\right\} 
(\check{a}(n)\cdot v_1)\otimes v_2+ \sum_{a(-n)\in \hB}\left\{[x_\ell(-p_\ell), \check{b}(k)], {a}(-n)\right\} 
({a}(-n)\cdot v_1)\otimes v_2 \\
&+ v_1\otimes [\theta (x_\ell(-p_\ell), \theta(\check{b}(k)]\cdot v_2 + (\delta_{k, p_\ell}([x_\ell, \check{b}] - p_\ell \langle x_\ell, \check{b}\rangle C)\cdot v_1)\otimes v_2\bigg)\\
& \qquad \mbox{where in the last sum only those $\check{b}$ appear which belong to $\fg_{p_\ell}$ and $\hB_{p_\ell}$ is an orthonormal basis of $\fg_{\underline{p}_\ell}$}\\
&=-\sum_{\ell; b(-k)\in \hB}\frac{1}{p_\ell}x_1(-p_1)\wedge \cdot \wedge b(-k)^{[\ell]}\wedge \cdots \wedge x_q(-p_q)\otimes\\
&\bigg(\sum_{a(-n)\in \hB}\left\{\left[a(-n), x_\ell(-p_\ell)\right],b(-k)\right\}(\check{a}(n)\cdot v_1)\otimes v_2 +\sum_{a(-n)\in \hB}
\left\{\left[\check{a}(n), x_\ell(-p_\ell)\right],b(-k)\right\}(a(-n)\cdot v_1)\otimes v_2\\
&+\left(\delta_{k,p_\ell}\left(\left[x_\ell, \check{b}\right]-p_\ell\langle x_\ell,\check{b}\rangle C \right)\cdot v_1\right) \otimes v_2
+ v_1 \otimes \left[\theta (x_\ell(-p_\ell)),\theta (\check{b}(k))\right]\cdot v_2\bigg), \mbox{ by Lemma \ref{lem_herm2}}
\\
&=-\sum_{\ell; a(-n)\in \hB}\frac{1}{p_\ell}\left(ad^\ell_{a(-n)}\vec{x}(-\vec{p})\right)\otimes \left( \check{a}(n) v_1\right) \otimes v_2 - \sum_{\ell; a(-n)\in \hB}\frac{1}{p_\ell}\left(\overline{ad}^\ell_{\check{a}(n)}\vec{x}(-\vec{p})\right)\otimes \left( a(-n)\cdot v_1\right) \otimes v_2\\
&+\frac{qc}{m}\vec{x}(-\vec{p})\otimes  v_1 \otimes v_2
-\sum_{\ell; b_0\in \Bo}\frac{1}{p_\ell}ad^\ell_{b_0} \vec{x}(-\vec{p})\otimes (\check{b}_0\cdot v_1)\otimes v_2\\
&-\sum_{\ell; a(-n)\in \hB}\frac{1}{p_\ell}\left(ad^\ell_{a(-n)}\vec{x}(-\vec{p})\right)\otimes v_1 \otimes \left((z\bar{z})^{p_\ell} z^n \check{a}\right)\cdot v_2 - \sum_{\ell; a(-n)\in \hB}\frac{1}{p_\ell}\left(\overline{ad}^\ell_{\check{a}(n)}\vec{x}(-\vec{p})\right)\otimes v_1 \otimes \left((z\bar{z})^{p_\ell-n}\bar{z}^n a\right)\cdot v_2\\
&-\sum_{\ell; b_0\in \mathcal{B}_{\underline{0}}}\frac{1}{p_\ell}\left(ad^\ell_{b_o} \vec{x}(-\vec{p})\right)\otimes v_1\otimes  \left( (z\bar{z})^{p_\ell}\check{b}_0\right)\cdot v_2 
+\sum_{a(-n)\in \hB} \frac{1}{n} \vec{x}(-\vec{p})\otimes a(-n)\check{a}(n)\cdot v.
\end{align*}

\begin{align*}
&\eqref{eqn7.5.3} +\eqref{eqn7.5.4} +\eqref{eqn7.5.7}\\
&=-\frac{1}{2}\sum_{\ell; b(-k), a(-n)\in\hat{\mathcal{B}}}
\frac{p_\ell+k}{p_\ell\cdot k}\left[ a (-n),b(-k)\right]\wedge \left(i^\ell_{a(-n)} \vec{x}(-\vec{p})\right)\otimes \check{b}(k)\cdot v
+\frac{1}{2}\sum_{\ell; b(-k)\in\hat{\mathcal{B}}}
\frac{p_\ell+k}{p_\ell\cdot k} ad^j_{b(-k)}\vec{x}(-\vec{p})\otimes \check{b}(k)\cdot v\\
&=\sum_{\ell; b(-k)\in\hat{\mathcal{B}}}\frac{p_\ell+k}{p_\ell\cdot k} ad^\ell_{b(-k)} \vec{x}(-\vec{p})\otimes \check{b}(k)\cdot v.\\
\end{align*}

\begin{align*}
&\eqref{eqn7.5.1} +\eqref{eqn7.5.6} =\\
&\frac{1}{4} \sum_{1\leq \ell\leq q; 1\leq j\leq q+1; b(-k), a(-n)\in\hat{\mathcal{B}}}
\frac{p_j(\ell)+n}{p_j(\ell)\cdot n}i_{a(-n)}ad^j_{a(-n)}
\left( b(-k)\wedge \overline{ad}^\ell_{\check{b}(k)}\vec{x}(-\vec{p})\right)\otimes v\\
&+\frac{1}{4} \sum_{j; b(-k), a(-n)\in \hB}\frac{n+p_j}{n\cdot p_j}b(-k)\wedge \overline{ad}_{\check{b}(k)} \left(i_{a(-n)}ad^j_{a(-n)}\vec{x}(-\vec{p})\right)\otimes v\\
&=\frac{1}{4} \sum_{\ell; b(-k), a(-n)\in\hat{\mathcal{B}}} \frac{k+n}{kn}i_{a(-n)}\left([a(-n), b(-k)] \wedge  \overline{ad}^\ell_{\check{b}(k)}\vec{x}(-\vec{p})\right)\otimes v \\
&+ \frac{1}{4} \sum_{\ell\neq j; b(-k), a(-n)\in\hat{\mathcal{B}}}  \frac{p_j+n}{p_jn}i_{a(-n)}\left(b(-k) \wedge \ad^j_{a(-n)}\overline{ad}^\ell_{\check{b}(k)}\vec{x}(-\vec{p})\right)\otimes v\\
&+ \frac{1}{4} \sum_{\ell; b(-k), a(-n)\in\hat{\mathcal{B}}}  \frac{p_\ell-k+n}{(p_\ell-k)n}i_{a(-n)}\left(b(-k) \wedge \ad^\ell_{a(-n)}\overline{ad}^\ell_{\check{b}(k)}\vec{x}(-\vec{p})\right)\otimes v \\
&+ \frac{1}{4} \sum_{j; b(-k), a(-n)\in \hB}\frac{n+p_j}{n\cdot p_j}b(-k)\wedge \overline{ad}_{\check{b}(k)} \left(i_{a(-n)}ad^j_{a(-n)}\vec{x}(-\vec{p})\right)\otimes v\\
&=-\frac{1}{4} \sum_{\ell; b(-k), a(-n)\in\hat{\mathcal{B}}} \frac{k+n}{kn}
\left([a(-n), b(-k)] \wedge i_{a(-n)} \overline{ad}^\ell_{\check{b}(k)}\vec{x}(-\vec{p})\right)\otimes v \\
&+ \frac{1}{4} \sum_{\ell\neq j; b(-k)\in\hat{\mathcal{B}}} 
\frac{p_j+k}{p_jk} \left(\ad^j_{b(-k)}\overline{ad}^\ell_{\check{b}(k)}\vec{x}(-\vec{p})\right)\otimes v\\
& - \frac{1}{4} \sum_{\ell\neq j; b(-k), a(-n)\in\hat{\mathcal{B}}} \frac{p_j+n}{p_jn}
b(-k) \wedge \left(i_{a(-n)} \ad^j_{a(-n)}\overline{ad}^\ell_{\check{b}(k)}\vec{x}(-\vec{p})\right)\otimes v \\
 &+ \frac{1}{4} \sum_{\ell; b(-k)\in\hat{\mathcal{B}}}  \frac{p_\ell}{(p_\ell-k)k} \left(\ad^\ell_{b(-k)}\overline{ad}^\ell_{\check{b}(k)}\vec{x}(-\vec{p})\right)\otimes v  \\
 &  - \frac{1}{4} \sum_{\ell; b(-k), a(-n)\in\hat{\mathcal{B}}} \frac{p_\ell -k+n}{(p_\ell -k)n}
b(-k) \wedge \left(i_{a(-n)} \ad^\ell_{a(-n)}\overline{ad}^\ell_{\check{b}(k)}\vec{x}(-\vec{p})\right)\otimes v  \\
&+ \frac{1}{4} \sum_{j; b(-k), a(-n)\in \hB}\frac{n+p_j}{n\cdot p_j}b(-k)\wedge \overline{ad}_{\check{b}(k)}
\left(i_{a(-n)}ad^j_{a(-n)}\vec{x}(-\vec{p})\right)\otimes v\\
&=\frac{1}{4}\sum_{j\neq \ell; b(-k)\in \hB} \frac{k+p_j}{k\cdot p_j}\ad^j_{b(-k)}\overline{ad}^\ell_{\check{b}(k)} \vec{x}(-\vec{p})\otimes v\\
\end{align*}
\begin{align*} 
&+\frac{1}{4}\sum_{\ell; b(-k)\in \hB} \frac{p_\ell}{(p_\ell-k)k} \ad^\ell_{b(-k)}\overline{ad}^\ell_{\check{b}(k)} \vec{x}(-\vec{p})\otimes v\\
&+\frac{1}{4}\sum_{j\neq \ell; b(-k)\in \hB} \frac{k+p_j}{kp_j}ad^j_{b(-k)}\overline{ad}^\ell_{\check{b}(k)}\vec{x}(-\vec{p})\otimes v\\
&+\frac{1}{4}\sum_{\ell\neq j; b(-k)\in \hB}\frac{p_j+p_\ell-k}{p_j(p_\ell-k)} x_1(-p_1) \wedge \dots \wedge b(-k)^{[\ell]}\wedge \dots \wedge \left(ad_{x_j(-p_j)}\overline{ad}_{\check{b}(k)}(x_\ell(-p_\ell))\right)^{[j]}\wedge \dots \wedge x_q(-p_q)\otimes v\\
&-\frac{1}{4}\sum_{\ell\neq j; i\notin \{j, \ell\}; b(-k)\in \hB}\frac{p_j+p_i}{p_j\cdot p_i} x_1(-p_1) \wedge \dots \wedge \left(\overline{ad}_{\check{b}(k)} x_\ell(-p_\ell)\right)^{[\ell]}
\wedge \dots \wedge [x_i(-p_i), x_j(-p_j)]^{[j]}\wedge \dots \wedge b(-k)^{[i]}\wedge \dots \wedge x_q(-p_q)\otimes v\\
&+\frac{1}{4} \sum_{\ell; b(-k)\in \hB} \frac{p_\ell}{k(p_\ell-k)} {ad}^\ell_{b(-k)}\overline{ad}^\ell_{\check{b}(k)} \vec{x}(-\vec{p})\otimes v\\
&-\frac{1}{4}\sum_{\ell\neq j; b(-k)\in \hB} \frac{p_j+p_\ell-k}{p_j(p_\ell-k)} x_1(-p_1) \wedge \dots \wedge \left(ad_{x_j(-p_j)}\overline{ad}_{\check{b}(k)}(x_\ell(-p_\ell))\right)^{[\ell]}\wedge \dots  \wedge b(-k)^{[j]}\wedge \dots 
\wedge x_q(-p_q)\otimes v\\
&+\frac{1}{4}\sum_{j\neq \ell; b(-k)\in \hB}\frac{p_j+p_\ell}{p_j\cdot p_\ell}\, \overline{ad}_{\check{b}(k)}\left( x_1(-p_1) \wedge \dots  \wedge \left(ad_{x_\ell(-p_\ell)}x_j(-p_j)\right)^{[j]}\wedge \dots  \wedge b(-k)^{[\ell]}\wedge
\dots  \wedge x_q(-p_q)\right)\otimes v\\
&= \frac{1}{2}\sum_{j\neq \ell; b(-k)\in \hB} \frac{k+p_j}{k\cdot p_j}\ad^j_{b(-k)}\overline{ad}^\ell_{\check{b}(k)} \vec{x}(-\vec{p})\otimes v
+ \frac{1}{2}\sum_{\ell; b(-k)\in \hB} \frac{p_\ell}{k (p_\ell -k)}\ad^\ell_{b(-k)}\overline{ad}^\ell_{\check{b}(k)} \vec{x}(-\vec{p})\otimes v\\
&+\frac{1}{4}\sum_{\ell\neq j; b(-k)\in \hB}\frac{p_j+p_\ell-k}{p_j(p_\ell-k)} x_1(-p_1) \wedge \dots \wedge b(-k)^{[\ell]}\wedge \dots \wedge \left(ad_{x_j(-p_j)}\overline{ad}_{\check{b}(k)}(x_\ell(-p_\ell))\right)^{[j]}\wedge \dots \wedge x_q(-p_q)\otimes v\\
&-\frac{1}{4}\sum_{\ell\neq j; b(-k)\in \hB}\frac{p_j+p_\ell-k}{p_\ell(p_j-k)} x_1(-p_1) \wedge \dots \wedge b(-k)^{[\ell]}\wedge \dots \wedge \left(ad_{x_\ell(-p_\ell)}\overline{ad}_{\check{b}(k)}(x_j(-p_j))\right)^{[j]}\wedge \dots \wedge x_q(-p_q)\otimes v\\
&+\frac{1}{4}\sum_{\ell\neq j; b(-k)\in \hB} \frac{p_j+p_\ell}{p_j\cdot p_\ell} x_1(-p_1) \wedge \dots  \wedge \left(\overline{ad}_{\check{b}(k)}\ad_{x_j(-p_j)}(x_\ell(-p_\ell))\right)^{[\ell]}\wedge \dots  \wedge b(-k)^{[j]}
\wedge x_q(-p_q)\otimes v\\
&=\frac{1}{2}\sum_{j\neq \ell; b(-k)\in \hB} \frac{k+p_j}{k\cdot p_j}\ad^j_{b(-k)}\overline{ad}^\ell_{\check{b}(k)} \vec{x}(-\vec{p})\otimes v\\
&+ \frac{1}{2}\sum_{\ell; b(-k)\in \hB} \frac{p_\ell}{k (p_\ell -k)}\ad^\ell_{b(-k)}\overline{ad}^\ell_{\check{b}(k)} \vec{x}(-\vec{p})\otimes v
+\frac{1}{4} \sum_{\ell \neq j; a(-n)\in \hB} \frac{p_j+n}{n \cdot p_j}\, \overline{ad}^\ell_{\check{a}(n)} \ad^j_{a(-n)}
\vec{x}(-\vec{p})\otimes v\\
&-\frac{1}{4}\sum_{\ell\neq j; a(-n)\in \hB}\frac{p_\ell+n}{n\cdot p_\ell} x_1(-p_1)\wedge \dots \wedge\left(\overline{ad}_{\check{a}(n)}(x_j(-p_j))\right)^{[\ell]}\wedge \dots \wedge  \left(ad_{a(-n)}x_\ell(-p_\ell)\right)^{[j]}\wedge \cdots
x_q(-p_q)\otimes v\\
&+\frac{1}{4}\sum_{\ell\neq j; b(-k)\in \hB} \frac{p_j+p_\ell}{p_j\cdot p_\ell} x_1(-p_1) \wedge \dots  \wedge \left(\overline{ad}_{\check{b}(k)}\ad_{x_j(-p_j)}(x_\ell(-p_\ell))\right)^{[\ell]}\wedge \dots  \wedge b(-k)^{[j]}\wedge \dots
\wedge x_q(-p_q)\otimes v,\,\,\text{by Lemma \ref{newlem7.2.9}}\\
&=\sum_{j\neq \ell; b(-k)\in \hB} \frac{p_\ell-k}{k\cdot p_\ell}\ad^j_{b(-k)}\overline{ad}^\ell_{\check{b}(k)} \vec{x}(-\vec{p})\otimes v
+ \frac{1}{2}\sum_{\ell; b(-k)\in \hB} \frac{p_\ell}{k (p_\ell -k)}\ad^\ell_{b(-k)}\overline{ad}^\ell_{\check{b}(k)} \vec{x}(-\vec{p})\otimes v\\
& -
\frac{1}{2} \sum_{j\neq \ell; b_0\in \mathcal{B}_{\underline{0}}} \frac{p_\ell+p_j}{p_\ell p_j}ad^j_{b_0}
{ad}^\ell_{\check{b}_0}\vec{x}(-\vec{p})\otimes v,\,\, \text{by following Lemma \ref{lem6.2.8}}.
\end{align*}
This completes the proof of the lemma. 
\end{proof}

\begin{lemma}\label{newlem7.2.9}
For $j\neq \ell$ and $\vec{x}(-\vec{p})= x_1(-p_1)\wedge\dots \wedge x_q(-p_q)\in  \wedge ^{q}\left(L^-(\fg, \sigma)\right)$, we have
\begin{align*}
&\sum_{b(-k)\in\hat{\mathcal{B}}}\, \frac{p_j+p_\ell-k}{p_j(p_\ell-k)} x_1(-p_1) \wedge \dots \wedge b(-k)^{[\ell]}\wedge \dots \wedge \left(ad_{x_j(-p_j)}\overline{ad}_{\check{b}(k)}(x_\ell(-p_\ell))\right)^{[j]}\wedge \dots \wedge x_q(-p_q)\\
&=\sum_{a(-n)\in\hat{\mathcal{B}}}\, \frac{p_j+n}{p_j\cdot n} \overline{ad}^\ell_{\check{a}(n)}\ad^j_{a(-n)}\vec{x}(-\vec{p}).
\end{align*}
\end{lemma}
\begin{proof}
\begin{align*}
&\sum_{b(-k)\in\hat{\mathcal{B}}}\, \frac{p_j+p_\ell-k}{p_j(p_\ell-k)} x_1(-p_1) \wedge \dots \wedge b(-k)^{[\ell]}\wedge \dots \wedge \left(ad_{x_j(-p_j)}\overline{ad}_{\check{b}(k)}(x_\ell(-p_\ell))\right)^{[j]}\wedge \dots \wedge x_q(-p_q)\\
&=\sum_{b(-k), a(-n)\in\hat{\mathcal{B}}}\, \frac{p_j+p_\ell-k}{p_j(p_\ell-k)} x_1(-p_1) \wedge \dots \wedge \{[\check{b}(k), x_\ell(-p_\ell)], a(-n)\}b(-k)^{[\ell]}\wedge \dots \wedge \left(ad_{x_j(-p_j)}(a(-n))\right)^{[j]}\wedge \dots \wedge x_q(-p_q)\\
&=\sum_{b(-k), a(-n)\in\hat{\mathcal{B}}}\, \frac{p_j+p_\ell-k}{p_j(p_\ell-k)} x_1(-p_1) \wedge \dots \wedge \{[\check{a}(n), x_\ell(-p_\ell], b(-k)\}b(-k)^{[\ell]}\wedge \dots \wedge \left(ad_{a(-n)}x_j(-p_j)\right)^{[j]}\wedge \dots \wedge x_q(-p_q),\\&\qquad\qquad\text{by Lemma \ref{lem_herm2}}\\
&=\sum_{a(-n)\in\hat{\mathcal{B}}}\, \frac{p_j+n}{p_j\cdot n} x_1(-p_1) \wedge \dots \wedge \left(\overline{ad}_{\check{a}(n)}x_\ell(-p_\ell)\right)^{[\ell]}\wedge \dots \wedge \left(ad_{a(-n)} x_j(-p_j)\right)^{[j]}\wedge \dots \wedge x_q(-p_q).
\end{align*}
This proves the lemma.
\end{proof}

\begin{lemma}\label{lem6.2.8}
For $\vec{x}(-\vec{p})= x_1(-p_1)\wedge\dots \wedge x_q(-p_q) \in  \wedge ^{q}\left(L^-(\fg, \sigma)\right)$, we have
\begin{align*}
&\sum_{j\neq \ell; b(-k)\in\hat{\mathcal{B}}}\left(\frac{1}{p_\ell}+\frac{1}{p_j}\right) x_1 (-p_1)\wedge \cdots \wedge b(-k)^{[j]}
\wedge \cdots \wedge\left[\left[\check{b}(k),x_j (-p_j)\right],x_\ell(-p_\ell)\right]_-^{[\ell]} \\
&\qquad\wedge \cdots \wedge x_q(-p_q)\\
&=-2\sum_{j\neq \ell; a(-n)\in\hat{\mathcal{B}}}\left(\frac{1}{p_\ell}+\frac{1}{p_j}\right)\, ad^j_{a(-n)}\,\overline{ad}^\ell_{\check{a}(n)}\vec{x}(-\vec{p})
-\sum_{j\neq \ell; b_0\in \CB_{\underline{0}}}\left(\frac{1}{p_\ell}+\frac{1}{p_j}\right)\, ad^j_{b_0}\,ad^\ell_{\check{b}_0} \vec{x}(-\vec{p}),
\end{align*}
where $[\,]_-$ denotes the projection to  $L^-(\fg, \sigma)$.
Similarly, 
\begin{align*}&\sum_{j\neq \ell; b(-k)\in\hat{\mathcal{B}}}\left(\frac{1}{p_\ell}+\frac{1}{p_j}\right) x_1 (-p_1)\wedge \cdots \wedge b(-k)^{[j]}
\wedge \cdots \wedge\left[x_j (-p_j),[\check{b}(k),x_\ell(-p_\ell)]\right]^{[\ell]} _-
\wedge \cdots \wedge x_q(-p_q)\\
&=-2\sum_{j\neq \ell; a(-n)\in\hat{\mathcal{B}}}\left(\frac{1}{p_j}+\frac{1}{p_\ell}\right)\, ad^j_{a(-n)}\,\overline{ad}^{\ell}_{\check{a}(n)} \vec{x}(-\vec{p})
-\sum_{j\neq \ell; b_0\in \CB_{\underline{0}}}
\left(\frac{1}{p_j}+\frac{1}{p_\ell}\right)\, ad^j_{b_0}\,ad^\ell_{\check{b}_0} \vec{x}(-\vec{p}).
\end{align*}
\end{lemma}
\begin{proof}
\begin{align*}
&\sum_{j\neq \ell; b(-k)\in\hat{\mathcal{B}}}\left(\frac{1}{p_\ell}+\frac{1}{p_j}\right) x_1 (-p_1)\wedge \cdots \wedge b(-k)^{[j]}
\wedge \cdots \wedge\left[\left[\check{b}(k),x_j (-p_j)\right],x_\ell(-p_\ell)\right]^{[\ell]}_- \\
&\qquad\wedge \cdots \wedge x_q(-p_q)\\
&=\sum_{j\neq \ell; a(-n), b(-k)\in\hat{\mathcal{B}}}\left(\frac{1}{p_\ell}+\frac{1}{p_j}\right)x_1(-p_1)\wedge \cdots \wedge b(-k)^{[j]} 
\wedge \cdots \wedge \left\{\left[\check{b}(k),x_j (-p_j)\right],a(-n)\right\}\\
&\qquad\qquad\left[a(-n),x_\ell(-p_\ell)\right]^{[\ell]}\wedge \cdots \wedge x_q(-p_q)\\
&=\sum_{j\neq \ell; a(-n), b(-k)\in\hat{\mathcal{B}}}\left(\frac{1}{p_\ell}+\frac{1}{p_j}\right)x_1(-p_1)\wedge \cdots \wedge b(-k)^{[j]} 
\wedge \cdots \wedge \left\{\left[\check{b}(k),x_j (-p_j)\right],\check{a}(n)\right\}\\
&\qquad\qquad
\left[\check{a}(n),x_\ell(-p_\ell)\right]^{[\ell]}\wedge \cdots \wedge x_q(-p_q)\\
&+\sum_{j\neq \ell; b_0\in \CB_{\underline{0}};  b(-k)\in\hat{\mathcal{B}}}\left(\frac{1}{p_\ell}+\frac{1}{p_j}\right)x_1(-p_1)\wedge \cdots \wedge b(-k)^{[j]} 
\wedge \cdots \wedge \left\{\left[\check{b}(k),x_j (-p_j)\right],b_0\right\}\\
&\qquad\qquad
\left[b_0,x_\ell(-p_\ell)\right]^{[\ell]}\wedge \cdots \wedge x_q(-p_q)\\
&=\sum_{j\neq \ell; a(-n)\in\hat{\mathcal{B}}}\left(\frac{1}{p_\ell}+\frac{1}{p_j}\right)x_1(-p_1)\wedge \cdots \wedge \left[x_j (-p_j),\check{a}(n)\right]^{[j]}_- 
\\
&\qquad\wedge \cdots \wedge \left[a(-n),x_\ell(-p_\ell)\right]^{[\ell]}
\wedge \cdots \wedge x_q(-p_q)\\
&+\sum_{j\neq \ell; a(-n)\in\hat{\mathcal{B}}}\left(\frac{1}{p_\ell}+\frac{1}{p_j}\right)x_1(-p_1)\wedge \cdots \wedge \left[x_j (-p_j),a(-n)\right]^{[j]}\\
&\qquad\qquad
\wedge \cdots \wedge \left[\check{a}(n),x_\ell(-p_\ell)\right]^{[\ell]}_-
\wedge \cdots \wedge x_q(-p_q)\\
&+\sum_{j\neq \ell; b_0\in \CB_{\underline{0}}}\left(\frac{1}{p_\ell}+\frac{1}{p_j}\right)x_1(-p_1)\wedge \cdots \wedge \left[x_j (-p_j),\check{b}_0\right]^{[j]} \wedge \cdots \wedge \left[b_0,x_\ell(-p_\ell)\right]^{[\ell]}\\
&\qquad\qquad
\wedge \cdots \wedge x_q(-p_q),\,\,\,\text{by Lemma  \ref{lem_herm2}}\\
&=-2\sum_{j\neq \ell; a(-n)\in\hat{\mathcal{B}}}\left(\frac{1}{p_\ell}+\frac{1}{p_j}\right)\, ad^j_{a(-n)}\,\overline{ad}^\ell_{\check{a}(n)}\vec{x}(-\vec{p})
-\sum_{j\neq \ell; b_o\in \CB_{\underline{0}}}\left(\frac{1}{p_\ell}+\frac{1}{p_j}\right)\, ad^j_{b_0}\,ad^\ell_{\check{b}_0} \vec{x}(-\vec{p}).
\end{align*}
\end{proof}

\begin{lemma}\label{6.2.10}
If $\vec{x}(-\vec{p})\otimes v \in \wedge^q (L^-(\fg,\sigma)) \otimes \mathscr{H}_c(\lambda)\otimes V(\vec{\mu})_{\vec{z}}
$ is $\fg^\sigma$-invariant, then
$$
\frac{1}{2}\sum_{1\leq j,\ell\leq q; b_0\in \CB_{\underline{0}}} \left(\frac{1}{p_j}+ \frac{1}{p_\ell}\right) ad^j_{b_0} ad^\ell_{\check{b}_0} \vec{x}(-\vec{p})\otimes v
+\sum_{\ell; b_0\in \CB_{\underline{0}}}
\frac{1}{p_\ell}ad^\ell_{b_0} \vec{x}(-\vec{p})\otimes \check{b}_0 \cdot v=0,$$
where $\fg_{\underline{0}}= \fg^\sigma$ is the $+1$ eigenspace with respect to the operator $\sigma$.
\end{lemma}
\begin{proof}
For any $x\in \fg_{\underline{0}}$ 
\begin{equation}\label{eqn7.9.2}
x\cdot \left( \sum_\ell \frac{1}{p_\ell}ad^\ell_x \vec{x}(-\vec{p}) \otimes v\right)=0.
\end{equation}

To prove this, observe that
\begin{align*}
x\cdot \left( \sum_\ell \frac{1}{p_\ell}ad^\ell_x \vec{x}(-\vec{p}) \otimes v\right)&=
\sum_{j,\ell}\frac{1}{p_\ell} ad^\ell_x ad^j_x \vec{x}(-\vec{p}) \otimes v + \sum_\ell \frac{1}{p_\ell} ad^\ell_x \vec{x}(-\vec{p})\otimes x\cdot v,\\
&\qquad \mbox{since $ad^\ell_x$  and $ad^j_x$ commute}\\
&= \left( \sum_\ell \frac{1}{p_\ell} ad^\ell_x\right) \cdot  \left( ad_x \vec{x}(-\vec{p}) \otimes v + \vec{x}(-\vec{p}) \otimes x\cdot v \right )\\
&= 0 , \qquad \mbox{since $\vec{x}(-\vec{p}) \otimes v $ is $\fg_{\underline{0}}$-invariant.}
\end{align*}
This proves \eqref{eqn7.9.2}. 

Thus, for any fixed $b_0$,
$$\sum_\ell \frac{1}{p_\ell}ad^\ell_{ \check{b}_0} \vec{x}(-\vec{p})\otimes {b}_0 \cdot v=
-\sum_\ell \frac{1}{p_\ell}\left( ad_{b_0}  ad^\ell_{\check{b}_0} \vec{x} (-\vec{p})\right)\otimes v.
$$
Thus, summing over $b_0$, we get the following equation:
\begin{equation}\label{eqn7.9.1}
\sum_{\ell; b_0\in \CB_{\underline{0}}}
 \frac{1}{p_\ell} ad^\ell_{\check{b}_0} \vec{x}(-\vec{p})\otimes {b}_0 \cdot v
= - \sum_{j, \ell; b_0\in \CB_{\underline{0}}}
 \frac{1}{p_\ell} ad ^j_{b_0} ad ^\ell_{\check{b}_0} \vec{x}(-\vec{p})\otimes v.
\end{equation}
Moreover,
\begin{align}
\frac{1}{2}&\sum_{j, \ell; b_0\in \CB_{\underline{0}}}
\left( \frac{1}{p_j}+\frac{1}{p_\ell}\right) \left(ad^j_{b_0} ad^\ell_{\check{b}_0}  \vec{x}(-\vec{p})\right) \otimes v\notag\\
&=\frac{1}{2}\sum_{j,\ell; b_0\in \CB_{\underline{0}}}
\frac{1}{p_j} \left(ad^j_{b_0} ad^\ell_{\check{b}_0}  \vec{x}(-\vec{p})\right) \otimes v
+\frac{1}{2}\sum_{j,\ell; b_0\in \CB_{\underline{0}}}
\frac{1}{p_\ell} \left(ad^j_{b_0} ad^\ell_{\check{b}_0}  \vec{x}(-\vec{p})\right) \otimes v\notag\\
&=\frac{1}{2}\sum_{j,\ell; b_0\in \CB_{\underline{0}}}
\frac{1}{p_\ell} \left(ad^\ell_{b_0} ad^j_{\check{b}_0} \vec{x}(-\vec{p})\right) \otimes v
+\frac{1}{2}\sum_{j,\ell; b_0\in \CB_{\underline{0}}}
\frac{1}{p_\ell} \left(ad^j_{b_0} ad^\ell_{\check{b}_0}  \vec{x}(-\vec{p})\right) \otimes v,\notag\\
& \qquad \text{by interchanging  $j$ and $\ell$ in the first factor},\notag\\
&=\frac{1}{2}\sum_{j,\ell; b_0\in \CB_{\underline{0}}}
\frac{1}{p_\ell} \left(ad^j_{\check{b}_0} ad^\ell_{{b}_0} \vec{x}(-\vec{p})\right) \otimes v
+\frac{1}{2}\sum_{j,\ell; b_0\in \CB_{\underline{0}}}
\frac{1}{p_\ell} \left(ad^j_{b_0} ad^\ell_{\check{b}_0}  \vec{x}(-\vec{p})\right) \otimes v, \label{eqn88}\\
& \qquad \text{since $\sum_{b_0\in \CB_{\underline{0}}} [b_o, [\check{b}_0, y]] = \sum_{b_0\in \CB_{\underline{0}}} [\check{b}_o, [{b}_0, y]]$ for any $y\in \fg$}.\notag
\end{align}
Combining  \eqref{eqn7.9.1} and \eqref{eqn88}, we get the lemma.
\end{proof}

\subsection{ Proof of  the Nakano's identity Theorem \ref{nakano}}\label{nakanoproof}
We first calculate $\overline{\Square}$. By Proposition \ref{prop7.5}
and Lemma \ref{lem7.6},
\begin{align*}
\overline{\Square}&\left(\vec{x}(-\vec{p})\otimes v\right)=\\
& +\sum_{b(-k)\in \hB} \frac{1}{k}\overline{ad}_{\check{b}(k)}\cdot\vec{x}(-\vec{p})\otimes b(-k)\cdot v
+\sum_{b(-k)\in \hB} \frac{1}{k}\vec{x}(-\vec{p})\otimes b(-k)\check{b}(k)\cdot v\\
&-\sum_{\ell; b(-k)\in \hB} \frac{1}{p_\ell}ad^\ell_{b(-k)}\vec{x}(-\vec{p})\otimes \left( \check{b}(k)\cdot v_1\right) \otimes v_2
-\sum_{\ell; b_0\in \CB_{\underline{0}}} \frac{1}{p_\ell}ad^\ell_{b_0}\vec{x}(-\vec{p})\otimes \left( \check{b}_0\cdot v_1\right) \otimes v_2\\
&-\sum_{\ell; b(-k)\in \hB}\frac{1}{p_\ell}\overline{ad}^\ell_{\check{b}(k)}\vec{x}(-\vec{p})\otimes \left( b(-k)\cdot v_1\right) \otimes v_2
+\frac{qc}{m} \vec{x}(-\vec{p})\otimes v_1 \otimes v_2\\
&-\sum_{\ell; b_0\in \CB_{\underline{0}}} \frac{1}{p_\ell}ad^\ell_{b_0} \vec{x}(-\vec{p})\otimes v_1 \otimes \left((z\bar{z})^{p_\ell} \check{b}_0\right) \cdot v_2\\
&-\sum_{\ell; a(-n)\in \hB}\frac{1}{p_\ell}\left(ad^\ell_{a(-n)}\vec{x}(-\vec{p})\right)\otimes v_1 \otimes \left((z\bar{z})^{p_\ell} z^n \check{a}\right)\cdot v_2 - \sum_{\ell; a(-n)\in \hB}\frac{1}{p_\ell}\left(\overline{ad}^\ell_{\check{a}(n)}\vec{x}(-\vec{p})\right)\otimes v_1 \otimes \left((z\bar{z})^{p_\ell-n}\bar{z}^n a\right)\cdot v_2\\
&+\sum_{\ell; b(-k)\in \hB}\left(\frac{1}{p_\ell}+\frac{1}{k}\right)ad^\ell_{b(-k)}\vec{x}(-\vec{p})\otimes \check{b}(k)\cdot v
+\sum_{j\neq \ell; b(-k)\in \hB}\left(\frac{1}{k}-\frac{1}{p_\ell}\right)ad^j_{b(-k)}\overline{ad}^\ell_{\check{b}(k)}\vec{x}(-\vec{p})\otimes v\\
& +\frac{1}{2}
\sum_{\ell; b(-k)\in \hB} {\left(\frac{1}{p_\ell-k}+\frac{1}{k}\right)}ad^\ell_{b(-k)}\overline{ad}^\ell_{\check{b}(k)} \vec{x}(-\vec{p})\otimes v
-\frac{1}{2}\sum_{j\neq \ell; b_0\in \CB_{\underline{0}}}
\left(\frac{1}{p_j} +\frac{1}{p_\ell}\right)ad^j_{b_0}ad^\ell_{\check{b}_0}\vec{x}(-\vec{p})\otimes v .
\end{align*}
Observe that for any $1\leq \ell\leq q$, 
$$\sum_{b(-k)\in \hB} {\left(\frac{1}{p_\ell-k}+\frac{1}{k}\right)}ad^\ell_{b(-k)}\overline{ad}^\ell_{\check{b}(k)} \vec{x}(-\vec{p})\otimes v =
\sum_{b(-k)\in \hB} {\left(\frac{1}{2k}\right)}ad^\ell_{b(-k)}\overline{ad}^\ell_{\check{b}(k)} \vec{x}(-\vec{p})\otimes v.$$
Thus, using the expression for $\Square$ as in Proposition \ref{newprop6.5}, we get the following:
\vskip1ex

For $ \vec{x}(-\vec{p})\otimes v_1 \otimes v_2 \in \wedge^q\left(L^-\left(\fg,\sigma\right)\right)\otimes \mathscr{H}_c(\lambda)\otimes V_{\vec{z}}(\vec{\mu}),$ we have
\begin{align*}
&\left( \overline{\Square}-\Square\right)\left(\vec{x}(-\vec{p})\otimes v_1 \otimes v_2\right)\\
&=\frac{qc}{m} \left( \vec{x}(-\vec{p})\otimes v_1 \otimes v_2\right)
+\sum_{\ell; b(-k)\in \hB}\frac{1}{p_\ell}\left(ad^\ell_{b(-k)} \vec{x}(-\vec{p})\right)\otimes v_1 \otimes \check{b}(k)
\cdot v_2\\
&-\sum_{\ell;  b_0\in \CB_{\underline{0}}} \frac{1}{p_\ell}ad^\ell_{b_0} \vec{x}(-\vec{p})\otimes v_1
 \otimes \left(({z}\bar{z})^{p_\ell} \check{b}_0 \right)\cdot v_2\\
 &-\sum_{\ell; a(-n)\in \hB}\frac{1}{p_\ell}\left(ad^\ell_{a(-n)}\vec{x}(-\vec{p})\right)\otimes v_1 \otimes \left((z\bar{z})^{p_\ell} z^n \check{a}\right)\cdot v_2 - \sum_{\ell; a(-n)\in \hB}\frac{1}{p_\ell}\left(\overline{ad}^\ell_{\check{a}(n)}\vec{x}(-\vec{p})\right)\otimes v_1 \otimes \left((z\bar{z})^{p_\ell-n}\bar{z}^n a\right)\cdot v_2\\ 
 &+\sum_{\ell; b(-k)\in \hB} \frac{1}{p_\ell}{ad}^\ell_{b(-k)}\overline{ad}^\ell_{\check{b}(k)}\vec{x}(-\vec{p})\otimes v
-\frac{1}{2} \sum_{j\neq \ell; b_0\in \CB_{\underline{0}}} \left(\frac{1}{p_j}+\frac{1}{p_\ell}\right)ad^j_{b_0} ad^\ell_{\check{b}_0} \vec{x}(-\vec{p})\otimes v\\
&+\sum_{\ell; b(-k)\in \hB} \frac{1}{p_\ell}\overline{ad}^\ell_{\check{b}(k)} \vec{x}(-\vec{p})\otimes v_1 \otimes b (-k)\cdot v_2
-\sum_{\ell; b_0\in \CB_{\underline{0}}} \frac{1}{p_\ell}ad^\ell_{b_0}\vec{x}(-\vec{p})\otimes \check{b}_0 \cdot v_1 \otimes v_2,
\end{align*}
since, for any $1\leq \ell\leq q$,
$$\sum_{b(-k)\in \hB}\frac{1}{p_\ell-k}\left(ad^\ell_{b(-k)} \overline{ad}^\ell_{\check{b}(k)}
\vec{x}(-\vec{p})\right)= \sum_{a(-n)\in \hB}\frac{1}{n}\left(ad^\ell_{a(-n)} \overline{ad}^\ell_{\check{a}(n)}
\vec{x}(-\vec{p})\right),$$
as can be seen by using Lemma  \ref{lem_herm2}.

Now, from the above expression of $ \overline{\Square}-\Square$, 
using Lemma \ref{6.2.10} and Proposition \ref{prop_casimir_id}, we get the following, which proves  the Nakano's identity Theorem \ref{nakano}.

For $\vec{x}(-\vec{p})\otimes v_1\otimes v_2 \in \left[ \wedge^q\left(L^-(\mathfrak{g},\sigma)\right)\otimes \mathscr{H}_c(\lambda) \otimes V_{\vec{z}}(\vec{\mu})\right]^{\mathfrak{g}^\sigma},$
\begin{align*}
\left(\overline{\Square}-\Square\right)& \left( \vec{x}(-\vec{p})\otimes v_1\otimes v_2\right) =
 \frac{(c+2\check{h})q}{m} \vec{x} (-\vec{p})\otimes v\\
&+ \sum_{\ell; b(-k)\in \hB} \frac{1}{p_\ell} ad^\ell_{b(-k)} \vec{x}(-\vec{p})\otimes v_1
\otimes  \left((1-(z\bar{z})^{p_\ell}) z^k \check{b}\right) \cdot v_2\\
&+ \sum_{\ell; b(-k)\in \hB }
 \frac{1}{p_\ell} \overline{ad}^\ell_{\check{b}(k)}\vec{x}(-\vec{p})\otimes v_1
\otimes \left((1-(z\bar{z})^{p_\ell-k}) 
(\bar{z})^{k} b\right) \cdot v_2\\
&+ \sum_{\ell; b_0\in \CB_{\underline{0}}}
 \frac{1}{p_\ell} ad^\ell_{b_0} \vec{x}(-\vec{p}) \otimes v_1 \otimes \left(()1-(z\bar{z})^{p_\ell}) \check{b}_0\right) \cdot v_2.
\end{align*}
Observe that we have obtained Theorem \ref{nakano} for {\it special} automorphisms $\sigma$ of $\fg$ since its proof given above uses 
Proposition \ref{prop_casimir_id}, which is proved only for special automorphisms $\sigma$.


\end{document}